\documentclass[12pt]{amsart}
\usepackage{mathrsfs, amssymb, array, amsmath, upgreek, amsfonts, mathtools, mathcomp, longtable}
\usepackage{amsthm}

\usepackage[matrix,arrow,curve]{xy}
\usepackage{hyperref}
\usepackage{ extpfeil, extarrows } 
\makeatletter
\newcommand*{\da@rightarrow}{\mathchar"0\hexnumber@\symAMSa 4B }
\newcommand*{\da@leftarrow}{\mathchar"0\hexnumber@\symAMSa 4C }
\newcommand*{\xdashrightarrow}[2][]{%
\mathrel{%
\mathpalette{\da@xarrow{#1}{#2}{}\da@rightarrow{\,}{}}{}%
}%
}
\newcommand{\xdashleftarrow}[2][]{%
\mathrel{%
\mathpalette{\da@xarrow{#1}{#2}\da@leftarrow{}{}{\,}}{}%
}%
}
\newcommand*{\da@xarrow}[7]{%
\sbox0{$\ifx#7\scriptstyle\scriptscriptstyle\else\scriptstyle\fi#5#1#6\m@th$}%
\sbox2{$\ifx#7\scriptstyle\scriptscriptstyle\else\scriptstyle\fi#5#2#6\m@th$}%
\sbox4{$#7\dabar@\m@th$}%
\dimen@=\wd0 %
\ifdim\wd2 >\dimen@
\dimen@=\wd2 %
\fi
\count@=2 %
\def\da@bars{\dabar@\dabar@}%
\@whiledim\count@\wd4<\dimen@\do{%
\advance\count@\@ne
\expandafter\def\expandafter\da@bars\expandafter{%
\da@bars
\dabar@ 
}%
}%
\mathrel{#3}%
\mathrel{%
\mathop{\da@bars}\limits
\ifx\\#1\\%
\else
_{\copy0}%
\fi
\ifx\\#2\\%
\else
^{\copy2}%
\fi
}%
\mathrel{#4}%
}
\makeatother

\newcommand{\ZZ}{\mathbb{Z}}
\newcommand{\PP}{\mathbb{P}}

\newcommand{\cA}{{\mathscr{A}}}

\newcommand{\cO}{{\mathscr{O}}}

\newcommand{\xref}[1]{\textup{\ref{#1}}}
\newcommand{\g}{{\mathrm{g}}}
\newcommand{\kk}{{\mathsf{k}}}
\newcommand{\bkk}{{\bar{\mathsf{k}}}}
\newcommand{\Gal}{{\mathbf{Gal}(\bkk/\kk)}}

\newcommand{\rG}{{\mathrm{G}}}

\newcommand{\rS}{{\mathrm{S}}}
\newcommand{\rSv}{{\mathrm{S}^\vee}}

\DeclareMathOperator{\Br}{\mathrm{Br}}

\DeclareMathOperator{\Ext}{\mathrm{Ext}}
\DeclareMathOperator{\Hom}{\mathrm{Hom}}
\DeclareMathOperator{\Sing}{\mathrm{Sing}}

\DeclareMathOperator{\rank}{\mathrm{rank}}
\DeclareMathOperator{\Pic}{\mathrm{Pic}}

\newcommand{\Aut}{\operatorname{Aut}}

\newcommand{\bF}{{\mathbf{F}}}
\newcommand{\bFv}{{\mathbf{F}^\vee}}
\newcommand{\bE}{{\mathbf{E}}}
\newcommand{\hbD}{\widehat{\mathbf{D}}}
\newcommand{\hbE}{\widehat{\mathbf{E}}}
\newcommand{\tcE}{\widetilde{\cE}}
\newcommand{\bH}{{\mathbf{H}}}
\newcommand{\bQ}{{\mathbf{Q}}}
\newcommand{\bX}{{\mathbf{X}}}
\newcommand{\bZ}{{\mathbf{Z}}}

\newcommand{\bK}{{\mathbf{K}}}

\newcommand{\hbX}{\widehat{\mathbf{X}}}
\newcommand{\hbZ}{\widehat{\mathbf{Z}}}

\DeclareMathOperator{\Gr}{\mathrm{Gr}}
\DeclareMathOperator{\CGr}{\mathrm{CGr}}
\newcommand{\tX}{{\tilde{X}}}

\DeclareMathOperator{\Bl}{\mathrm{Bl}}

\DeclareMathOperator{\OGr}{\mathrm{OGr}}

\newcommand{\cQ}{{\mathscr{Q}}}
\DeclareMathOperator{\OFl}{\mathrm{OFl}}

\DeclareMathOperator{\Fl}{\mathrm{Fl}}

\DeclareMathOperator{\Ker}{\mathrm{Ker}}

\DeclareMathOperator{\Ima}{\mathrm{Im}}

\newcommand{\cN}{{\mathscr{N}}}

\newcommand{\cU}{{\mathscr{U}}}

\newcommand{\rc}{{\mathrm{c}}}

\newcommand{\cE}{{\mathscr{E}}}

\newcommand{\barW}{{\overline{W}}}
\newcommand{\barX}{{\bar{X}}}
\newcommand{\barbH}{{\overline{\bH}}}
\newcommand{\barbX}{{\overline{\bX}}}

\DeclareMathOperator{\codim}{\mathrm{codim}}

\DeclareMathOperator{\LGr}{\mathrm{LGr}}
\DeclareMathOperator{\ITGr}{\mathrm{I_3Gr}}

\newcommand{\cL}{\mathscr{L}} 
\newcommand{\rF}{\mathrm{F}}

\DeclareMathOperator{\GTGr}{G_2Gr}
\DeclareMathOperator{\GL}{GL}

\DeclareMathOperator{\Sp}{Sp}

\newcommand{\fD}{{\mathfrak{D}}}

\newcommand{\Db}{\mathrm{D}^{\mathrm{b}}}

\makeatletter
\@addtoreset{equation}{subsection}
\makeatother

\theoremstyle{plain}

\newtheorem{theorem}{Theorem}[section]
\newtheorem{lemma}[theorem]{Lemma}

\newtheorem{proposition}[theorem]{Proposition}
\newtheorem{corollary}[theorem]{Corollary}

\theoremstyle{definition}

\newtheorem*{definition*}{Definition}
\newtheorem{example-remark}{Remark-Example}

\newtheorem*{notation*}{Notation}

\newtheorem{remark}[theorem]{Remark}

\title{Rationality of Mukai varieties over non-closed fields}
\author{Alexander Kuznetsov}
\thanks{The authors were partially supported by the HSE University Basic Research Program, Russian Academic Excellence Project ``5-100''.}
\address{\parbox{0.9\textwidth}{Steklov Mathematical Institute of Russian Academy of Sciences, Moscow, Russia 
\\[1pt]
Laboratory of Algebraic Geometry, NRU HSE, Moscow, Russia\\}}

\email{akuznet@mi-ras.ru}

\author{Yuri Prokhorov}

\address{\parbox{0.9\textwidth}{Steklov Mathematical Institute of Russian Academy of Sciences, Moscow, Russia 
\\[1pt]
Laboratory of Algebraic Geometry, NRU HSE, Moscow, Russia
\\[1pt]
Department of Algebra, Moscow State University, Moscow, Russia
\\}}

\email{prokhoro@mi-ras.ru}

\date{}

\begin{document}

\begin{abstract}
We discuss birational properties of Mukai varieties, i.e., 
of higher-dimensional analogues of prime Fano threefolds of genus~$g \in \{7,8,9,10\}$ 
over an arbitrary field~$\kk$ of zero characteristic.
In the case of dimension $n \ge 4$ we prove that these varieties are $\kk$-rational if and only if they have a $\kk$-point
except for the case of genus~$9$, where we assume $n \ge 5$.
Furthermore, we prove that Mukai varieties of genus $g \in \{7,8,9,10\}$ and dimension $n \ge 5$ 
contain cylinders if they have a $\kk$-point.
Finally, we prove that the embedding $X \hookrightarrow \Gr(3,7)$ for prime Fano threefolds of genus~$12$ 
is defined canonically over any field and use this to give a new proof of the criterion of rationality.
\end{abstract}

\maketitle
\tableofcontents

\section{Introduction}

A {\sf Mukai variety} 
is a (smooth) Fano variety of geometric Picard number~1 and coindex~3
which is not a form of a complete intersection in a weighted projective space.
Over an algebraically closed field of characteristic zero such varieties have been classified by Mukai in~\cite{Mukai92}.

The main invariant of a Mukai variety~$X$ is the {\sf genus}~$\g(X)$ defined by the formula 
\begin{equation*}
2\g(X) - 2 = H^n,
\end{equation*}
where $H$ is the ample generator of~$\Pic(X)$ and~$n = \dim(X)$.
The following table lists {\sf maximal Mukai varieties}, their genera and dimensions:
\begin{equation*}
\begin{array}{|c|r|c|}
\hline 
g & \multicolumn{1}{c|}{\bX_{2g-2}} & \dim(\bX_{2g-2}) \\
\hline 
6 & \CGr(2,5) \cap Q \subset \PP^{10} & 6 \\
\hline 
7 & \OGr_+(5,10) \subset \PP^{15} & 10 \\
\hline 
8 & \Gr(2,6) \subset \PP^{14} & 8 \\
\hline 
9 & \LGr(3,6) \subset \PP^{13} & 6 \\
\hline 
10 & \GTGr(2,7) \subset \PP^{13} & 5 \\
\hline 
12 & \ITGr(3,7) \subset \PP^{13} & 3 \\
\hline
\end{array}
\leqno{(\mathrm{M})}
\end{equation*}
Here 
\begin{itemize}
\item 
$\CGr(2,5) \cap Q$ is a transverse intersection of the cone over $\Gr(2,5)$ with a quadric,
\item 
$\OGr_+(5,10)$ is a connected component of the Grassmannian of isotropic $5$-dimensional subspaces 
in a 10-dimensional vector space endowed with a non-degenerate quadratic~form,
\item 
$\Gr(2,6)$ is the Grassmannian of $2$-dimensional subspaces in a $6$-dimensional vector space,
\item 
$\LGr(3,6)$ is the Grassmannian of Lagrangian ($3$-dimensional) subspaces in a $6$-dimensional symplectic vector space,
\item 
$\GTGr(2,7)$ is the adjoint Grassmannian of the simple algebraic group~$\rG_2$, and 
\item 
$\ITGr(3,7)$ is the Grassmannian of $3$-dimensional subspaces in a $7$-dimensional vector space
isotropic for a (sufficiently general) triple of skew-symmetric forms.
\end{itemize}
Note that for $g \in \{7,8,9,10\}$ the maximal varieties are homogeneous (and in particular rigid),
while for $g \in \{6,12\}$ they vary in non-trivial moduli spaces (of dimensions~$25$ and~$6$, respectively).

\begin{theorem}[{\cite{Mukai92}}]
\label{theorem-mukai}
If $X$ is a Mukai variety of genus~$g$ over an algebraically closed field of characteristic zero
then $g \in \{6,7,8,9,10,12\}$ and there is an embedding $X \hookrightarrow \bX_{2g-2}$ 
into one of the maximal varieties from the list~$(\text{\rm{M}})$ such that $X$ is a transverse linear section of $\bX_{2g-2}$.
\end{theorem}

An easy consequence of the Mukai's theorem is that over an arbitrary field~$\kk$ of characteristic zero
any Mukai variety is a $\kk$-form of a linear section of one of the maximal varieties.

The goal of this paper is to investigate birational properties of Mukai varieties over arbitrary fields.
The case of Mukai \emph{threefolds} has been studied in~\cite{KP19}, where in the case $g \in \{7,9,10,12\}$
we established for them criteria of unirationality and rationality over an arbitrary field~$\kk$ of characteristic zero
(for $g \in \{6,\, 8\}$ Mukai threefolds are irrational even over an algebraically closed field).

In this paper we consider Mukai varieties of dimension~$n \ge 4$.
The main result of the paper is the following

\begin{theorem}
\label{theorem:main}
Let $\kk$ be an arbitrary field of characteristic zero.
Let $X$ be a Mukai variety of genus~\mbox{$g = \g(X)$} and dimension~\mbox{$n = \dim(X)$} such that
\begin{itemize}
\item 
either $g \in \{7,\, 8,\, 10\}$ and $n \ge 4$,
\item 
or $g = 9$ and $n \ge 5$.
\end{itemize}
Then the following conditions are equivalent:
\begin{enumerate}
\item 
\label{thm:item:rationality}
$X$ is~$\kk$-rational;
\item 
\label{thm:item:unirationality}
$X$ is~$\kk$-unirational;
\item 
\label{thm:item:points}
$X(\kk) \ne \varnothing$.
\end{enumerate}
\nopagebreak
In the case $(g, n) = (9,\, 4)$ we have the equivalence~\ref{thm:item:unirationality} $\iff$ \ref{thm:item:points}.
\end{theorem}

Implications \ref{thm:item:rationality} $\implies$ \ref{thm:item:unirationality} $\implies$ \ref{thm:item:points} of the theorem are evident,
so the content of the paper is in the implication~\ref{thm:item:points} $\implies$ \ref{thm:item:rationality}
(or~\ref{thm:item:points} $\implies$ \ref{thm:item:unirationality} for $(g,n) = (9,4)$).
This implication is proved in Theorem~\ref{theorem:rationality} for~$g \in \{7,\, 8,\, 10\}$,
and Corollary~\ref{corollary:rationality-9} for~$g = 9$.

To prove the implication we use the Sarkisov link starting with the blowup of a point
for each of the maximal Mukai varieties~$\bX_{2g-2}$ of genus $g \in \{7,8,9,10\}$ over~$\bkk$.
These links are constructed in a uniform way for $g \in \{7,8,10\}$ in Theorem~\ref{theorem:qk} 
(see also Propositions~\ref{proposition:gr2v}, \ref{proposition:ogr510}, and~\ref{proposition:g2gr}),
and separately for $g = 9$ in Theorem~\ref{theorem:lg36}.

For~\mbox{$g \in \{7,\, 8,\, 10\}$} the links end with projective bundles over smooth $4$-dimensional Fano varieties
($\PP^4$, a quadric, and a quintic del Pezzo fourfold, respectively). 
We check in Theorem~\ref{theorem:sections} that after passing to a smooth linear section~$X \subset \bX_{2g-2}$ of dimension~$n$
we obtain a birational transformation (in general this is an example of a so-called ``bad link'') between~$X$ 
and a variety~$\tX^+$ with a morphism to~$\bX^+$ and general fiber~$\PP^{n-4}_\bkk$.
We prove in Theorem~\ref{theorem:rationality} that if we start with a Mukai variety~$X$ defined over~$\kk$ and a $\kk$-point, 
this transformation is also defined over~$\kk$ and the general fiber of the morphism~$\tX^+ \to \bX^+$ is isomorphic to~$\PP^{n-4}_\kk$;
this implies $\kk$-rationality of~$X$.

For $g = 9$ we use the link constructed in Theorem~\ref{theorem:lg36} in a similar way 
to show that any Mukai variety~$X$ of genus~$9$ and dimension $n$ with $X(\kk) \ne \varnothing$
is birational to a complete intersection~\mbox{$X^+ \subset \PP^6$} of~\mbox{$6 - n$} quadrics 
containing a distinguished Veronese surface $\rS \subset X^+$ 
such that the divisor $X^+ \cap \langle S \rangle = X^+ \cap \PP^5$ is $\kk$-rational, see Theorem~\ref{theorem:mukai-9}.
For $n \ge 5$ this implies $\kk$-rationality of~$X$ and for $n = 4$ this implies its $\kk$-unirationality, 
see Corollary~\ref{corollary:rationality-9}.
We expect that $\kk$-unirational Mukai fourfolds of genus~$9$
are not~$\kk$-rational in general,
however we establish for them a sufficient condition of $\kk$-rationality, Corollary~\ref{corollary:rationality-9-4}.

In Appendix~\ref{subsection:cylinders} we use the above results to prove that Mukai varieties of genus~\mbox{$g \in \{7,\, 8,\, 9,\, 10\}$}
and dimension~$n \ge 5$ have cylinders, see 
Proposition~\ref{proposition:cylinders}.

\medskip

The constructed birational transformations can be also applied to Mukai threefolds.

In the case $(g, n) = (7,\, 3)$ we obtain a birational transformation from~$X$ 
to a Springer-type resolution of a special singular quartic threefold in~$\PP^4$;
considering another Springer resolution we then extend this to a birational map between~$X$
and a (possibly singular) quintic del Pezzo threefold.
This gives a more direct than in~\cite{KP19} proof of $\kk$-rationality of~$X$, see Remark~\ref{remark:x-3-7} for details.

Analogously, in the case $(g, n) = (8,\, 3)$ we obtain a birational transformation between~$X$ and a cubic threefold (Remark~\ref{remark:x-3-8})
and in the case~$(g, n) = (10,\, 3)$ a birational transformation between~$X$ 
and a sextic del Pezzo fibration over~$\PP^1$ (Remark~\ref{remark:x-3-10}).

\medskip 

Finally, in the case $(g, n) = (12,\, 3)$ we prove that the embedding $X \hookrightarrow \Gr(3,7)$ constructed by Mukai
is defined over any field of characteristic zero and is canonical (Corollary~\ref{corollary:x22-equivariant}).
Using this we construct in Theorem~\ref{theorem:v22} a birational map between~$X$ and~$\PP^3$
which looks very similar to the general construction of Theorem~\ref{theorem:qk}.
Similarly to the case~\mbox{$(g, n) = (7,\, 3)$} this gives an alternative and more direct than in~\cite{KP19} proof of $\kk$-rationality
of prime Fano threefolds of genus~12 with~$\kk$-points.
We also apply these results to the derived category of coherent sheaves on~$X$, see Corollary~\ref{corollary:x22-db}.

\medskip

To finish the introduction we remind what is known about the case of Mukai varieties of genus~$6$ 
(so-called {\sf Gushel--Mukai varieties})
and dimension~\mbox{$n \ge 4$} which are not covered by Theorem~\ref{theorem:main}.

The case of Gushel--Mukai fourfolds is very hard and interesting already over~$\bkk$;
it is expected that a very general Gushel--Mukai fourfold is not $\bkk$-rational 
and there is a countable union of divisorial families (in the moduli space) of Gushel--Mukai fourfolds which are $\bkk$-rational, 
see~\cite{Prokhorov-1993c,DIM,K18GM,Debarre2020gushelmukai}, analogously to the case of cubic fourfolds.

In the remaining cases $g = 6$, $n \ge 5$, it is known that each Gushel--Mukai variety of dimension~$n \ge 5$ is $\bkk$-rational
(see~\cite[Proposition~4.2]{DK18}), but the question of its $\kk$-rationality is completely unclear.

We would like to thank the organizers of the conference ``Rationality of Algebraic Varieties''
on the Schiermonnikoog Island, where the idea of this paper was born.

\section{A birational transformation given by a family of quadrics}

Assuming that $\bX$ is a smooth projective variety ``birationally covered'' (in the sense described below) by a family of quadrics, 
we construct in~\S\S\ref{subsection:quadric-transformation-statement}--\ref{subsection:quadric-transformation-proof}
a birational transformation of~$\bX$ into a projective bundle.
This transformation is, essentially, a relative version of the birational map 
of a quadric and a projective space
\begin{equation*}
Q^m \xdashrightarrow{\hspace{27pt}} \PP^m
\end{equation*}
induced by a linear projection from a point $x_0 \in Q^m$;
it blows up the point~$x_0$ and then contracts all lines on~$Q^m$ passing through~$x_0$.
In~\S\S\ref{subsection:gr2v}--\ref{subsection:g2gr} we show that these results 
apply to maximal Mukai varieties~$\Gr(2,6)$, $\OGr_+(5,10)$, and~$\GTGr(2,7)$ of genus $g \in \{7,\, 8,\, 10\}$.

\subsection{The statement}
\label{subsection:quadric-transformation-statement}

Let $\bX \subset \PP(W)$ be a smooth projective variety and let~$x_0 \in \bX$ be a point.
We denote by~$\bH$ the restriction to~$\bX$ of the hyperplane class of $\PP(W)$ and consider it as a polarization.
Assume a projective scheme $\bX^+$ and a subscheme
\begin{equation*}
\cQ \subset \bX \times \bX^+
\end{equation*}
are given such that~$\cQ$ is a flat family of quadrics in~$\bX$ containing~$x_0$, i.e.,
for each point $x_+ \in \bX^+$ the fiber~$\cQ_{x_+} \subset \bX \subset \PP(W)$ of~$\cQ$ 
is a quadric in~$\bX$ containing the point~$x_0$.
We denote by 
\begin{equation*}
p_\cQ \colon \cQ \xlongrightarrow{\hspace{2em}} \bX
\qquad\text{and}\qquad 
q_\cQ \colon \cQ \xlongrightarrow{\hspace{2em}} \bX^+
\end{equation*}
the natural projections and by
\begin{equation}
\label{eq:section-cq}
s_{x_0} \colon \bX^+ \xlongrightarrow{\hspace{2em}} \cQ
\end{equation}
the section of~$q_\cQ$ given by the point~$x_0$.

Let $\rF_1(\bX,x_0)$ be the Hilbert scheme of lines passing through the point~$x_0$
and let~$\rF_1(\cQ/\bX^+,x_0)$ be the relative Hilbert scheme of lines in fibers of $q_\cQ \colon {} \cQ \to \bX^+$ 
passing through the point~$x_0$.
Let
\begin{equation*}
\cL(\bX,x_0) \subset \rF_1(\bX,x_0) \times \bX
\qquad\text{and}\qquad 
\cL(\cQ/\bX^+,x_0) \subset \rF_1(\cQ/\bX^+,x_0) \times_{\bX^+} \cQ
\end{equation*}
be the corresponding universal families of lines and let
\begin{equation*}
p_\cL \colon \cL(\bX,x_0) \xlongrightarrow{\hspace{2em}} \bX
\qquad\text{and}\qquad 
\tilde{p}_\cL \colon \cL(\cQ/\bX^+,x_0) \xlongrightarrow{\hspace{2em}} \cQ
\end{equation*}
be the natural projections.

To state the theorem we will need the following simple observation.

\begin{lemma}
\label{lemma:lines}
The morphism~$p_\cQ$ induces a morphism 
\begin{equation}
\label{eq:f1-morphism}
\rF_1(\cQ/\bX^+,x_0) \xlongrightarrow{\hspace{2em}} \rF_1(\bX,x_0)
\end{equation} 
such that
\begin{equation}
\label{eq:lines}
\cL(\cQ/\bX^+,x_0) \cong \rF_1(\cQ/\bX^+,x_0) \times_{\rF_1(\bX,x_0)} \cL(\bX,x_0)
\end{equation}
and $p_\cQ\bigl(\tilde{p}_\cL(\cL(\cQ/\bX^+,x_0))\bigr) = p_\cL\bigl(\cL(\bX,x_0)\bigr)$.
\end{lemma}

\begin{proof}
If $L \subset \cQ_{x_+}$ is a line through~$x_0$, then $p_\cQ(L) \subset \bX$ is a line through~$x_0$;
this defines the morphism~\eqref{eq:f1-morphism}.
Moreover, $p_\cQ \colon L \to p_\cQ(L)$ is an isomorphism; this proves~\eqref{eq:lines}.
The last statement of the lemma is obvious.
\end{proof}

\begin{theorem}
\label{theorem:qk}
Let $\bX \subset \PP(W)$ be a smooth projective variety, 
and let $\cQ \subset \bX \times \bX^+$ be a family of quadrics in~$\bX$ passing through a point~$x_0 \in \bX$. Set 
\begin{equation*}
n = \dim(\bX),
\qquad
m = \dim(\cQ/\bX^+).
\end{equation*}
Assume the following conditions hold:
\begin{enumerate}
\item\label{ass:f2k} 
The scheme $\bX^+$ is smooth, projective, connected, and $\uprho(\bX^+) = \uprho(\bX)$.

\item\label{ass:p2} 
The map $p_\cQ \colon \cQ \to \bX$ is birational.

\item\label{ass:qk} 
For any quadric $\cQ_{x_+} \subset X$ corresponding to a point $x_+ \in \bX^+$ we have $\cQ_{x_+} = \bX \cap \langle \cQ_{x_+} \rangle$.
\item\label{ass:f1} 
The scheme $\rF_1(\bX,x_0)$ is smooth, the scheme $\rF_1(\cQ/\bX^+,{x_0})$ is smooth over~$\rF_1(\bX,x_0)$ and connected, 
and 
\[
\dim(\rF_1(\bX,x_0)) < {}\dim(\rF_1(\cQ/\bX^+,{x_0})) = \dim(\bX) - 2.
\]
\end{enumerate}
Then there is a commutative diagram 
\begin{equation}
\label{diagram:typical}
\vcenter{\xymatrix@C=4em{
&
\hbD \ar@{^{(}->}[d] \ar[dl] \ar[dr]
\\
\cL(\bX,x_0) \ar@{^{(}->}[d] &
\hbX \ar[dl]^{\pi} \ar[dr]_{\pi_+} &
\rF_1(\cQ/\bX^+,{x_0}) \ar@{^{(}->}[d]
\\
\Bl_{x_0}(\bX) \ar[d]_\sigma \ar@{-->}[rr]^\psi \ar[dr]^\phi
&&
\PP_{\bX^+}(\cE) \ar[d]^{\sigma_+} \ar[dl]_{\phi_+}
\\
\bX &
\barbX &
\bX^+,\\
}}
\end{equation}
where 
\begin{itemize}
\item 
$\sigma$ is the blowup of $x_0$ with the exceptional divisor~$\bE \subset \Bl_{x_0}(\bX)$,

\item 
$\cE$ is a subbundle of rank~$m + 1$ in $\barW \otimes \cO_{\bX^+}$, where $\barW = W/\langle x_0 \rangle$,

\item 
$\sigma_+$ it the projective bundle morphism, 
\item 
$\phi$ is induced by the linear projection $\PP(W) \dashrightarrow \PP(\barW)$ from~$x_0$, 
\item 
$\phi_+$ is the morphism induced by the embedding of vector bundles~$\cE \hookrightarrow \barW \otimes \cO_{\bX^+}$,
\item 
$\barbX = \phi(\Bl_{x_0}(\bX)) = \phi_+(\PP_{\bX^+}(\cE)) \subset \PP(\barW)$,
\item 
$\pi$ is the blowup of $\cL(\bX,x_0)$, 
\item 
$\pi_+$ is the blowup of $\rF_1(\cQ/\bX^+,{x_0})$, 
\item 
$\hbD$ is the common exceptional divisor of $\pi$ and $\pi_+$,
\begin{equation*}
\hbD \cong \cL(\bX,x_0) \times_{\rF_1(\bX,x_0)} \rF_1(\cQ/\bX^+,{x_0}),
\end{equation*}
\item 
$\psi = \pi_+ \circ \pi^{-1} = \phi_+^{-1} \circ \phi$ is a small birational map \textup(flop or antiflip\textup). 
\end{itemize}

If\/ $\hbE \subset \hbX$ is the strict transform of\/ $\bE \subset \Bl_{x_0}(\bX)$,
the morphisms $\pi$ and~$\pi_+$ induce isomorphisms
\begin{equation}
\label{eq:hbe}
\hbE \cong \Bl_{\rF_1(\bX,x_0)}(\bE)
\qquad\text{and}\qquad 
\hbE \cong \PP_{\bX^+}(\cE_0),
\end{equation}
where $\cE_0 \subset \cE$ is a rank-$m$ subsheaf.

Moreover, the maps $\pi$, $\phi$, $\phi_+$, and~$\pi_+$ induce isomorphisms 
\begin{equation}
\label{eq:iso-bx-open}
\hbX \setminus \hbD \stackrel{\pi}{\cong}
\Bl_{x_0}(\bX) \setminus \cL(\bX,x_0) \stackrel{\phi}{\cong} 
\barbX \setminus \rF_1(\bX,x_0) \stackrel{\phi_+}{\cong} 
\PP_{\bX^+}(\cE) \setminus \rF_1(\cQ/\bX^+,x_0) \stackrel{\pi_+}{\cong}
\hbX \setminus \hbD,
\end{equation}
the restrictions of~$\pi$ and~$\pi_+$ to~$\hbD$ are the projections to the factors~$\cL(\bX,x_0)$ and~$\rF_1(\cQ/\bX^+,x_0)$, respectively,
the restriction of~$\phi$ to~$\cL(\bX,x_0)$ is the natural~$\PP^1$-fibration $\cL(\bX,x_0) \xrightarrow{\hspace{1em}} \rF_1(\bX,x_0)$,
and the restriction of~$\phi_+$ to~$\rF_1(\cQ/\bX^+,x_0)$ is the morphism~\eqref{eq:f1-morphism}, which is also a projective bundle.

Finally, the birational transformation~$\psi$ induces an isomorphism
\begin{equation}
\label{eq:iso-be-open}
\bE \setminus \rF_1(\bX,x_0) \cong \PP_{\bX^+}(\cE_0) \setminus \rF_1(\cQ/\bX^+,x_0).
\end{equation}
\end{theorem}

In the course of proof we will explain the construction of the vector bundle~$\cE$, the embeddings 
\begin{equation*}
\cL(\bX,x_0) \xhookrightarrow{\hspace{27pt}} \Bl_{x_0}(\bX),
\qquad 
\rF_1(\cQ/\bX^+,{x_0}) \xhookrightarrow{\hspace{27pt}} \PP_{\bX^+}(\cE),
\quad\text{and}\quad
\rF_1(\bX,x_0) \xhookrightarrow{\hspace{27pt}} \bE
\end{equation*}
and other details.
Before that we deduce from the theorem the following useful lemma.

\begin{lemma}
\label{lemma:linear-system}
In the setting of Theorem~\xref{theorem:qk}, assume that 
$K_\bX = -i\bH$ for some $i \in \ZZ$.
If the divisor class $-K_{\bX^+} + \rc_1(\cE)$ is very ample on~$\bX^+$ then the map 
\begin{equation*}
\sigma_+ \circ \psi \colon \Bl_{x_0}(\bX) \xdashrightarrow{\hspace{27pt}} \bX^+
\end{equation*}
is given by the linear system $\left|(i-m-1)\bH - (n-m-2)\bE\right|$.
\end{lemma}

\begin{proof}
Since $\psi$ is a small birational map, it identifies $\Pic(\Bl_{x_0}(\bX))$ and~$\Pic(\PP_{\bX^+}(\cE))$.
The canonical class of $\Bl_{x_0}(\bX)$ is equal to $-i\bH + (n-1)\bE$.
On the other hand, the canonical class of~$\PP_{\bX^+}(\cE)$ is equal to $K_{\bX^+} - \rc_1(\cE) - (m+1)\barbH$,
where $\barbH$ is the hyperplane class of~$\PP_{\bX^+}(\cE)$.
Finally, we have $\barbH = \bH - \bE$ since $\phi$ is induced by the linear projection from~$x_0$
and~$\phi_+$ is induced by the embedding $\cE \hookrightarrow \barW \otimes \cO$.
Combining all this, we obtain the equality
\begin{equation*}
- K_{\bX^+} + \rc_1(\cE) = i\bH - (n-1)\bE - (m+1)(\bH-\bE) = (i-m-1)\bH - (n-m-2)\bE.
\end{equation*}
If the left side of the equality is very ample, the result follows.
\end{proof}

\subsection{The proof}
\label{subsection:quadric-transformation-proof}

Since $\cQ \to \bX^+$ is a flat family of $m$-dimensional quadrics contained in~$\PP(W)$,
there is a vector subbundle $\tcE \hookrightarrow W \otimes \cO_{\bX^+}$ of rank~$m+2$ and an embedding
\begin{equation*}
\iota \colon \cQ \xhookrightarrow{\hspace{2em}} \PP_{\bX^+}(\tcE)
\end{equation*}
as a divisor of relative degree~$2$.
Note that the pullback to $\PP_{\bX^+}(\tcE)$ of the hyperplane class~$\bH$ of~$\PP(W)$ is a relative hyperplane class.
Therefore, there is a divisor class~$D$ on~$\bX^+$ such that
\begin{equation}
\label{eq:cq-class}
\iota(\cQ) \sim 2\bH + D
\end{equation} 
in $\Pic(\PP_{\bX^+}(\tcE))$.

The section~\eqref{eq:section-cq} can be also thought of as a section of~$\PP_{\bX^+}(\tcE)$;
it corresponds to an embedding~\mbox{$\cO \hookrightarrow \tcE$} of vector bundles.
We denote by 
\begin{equation*}
\cE = \tcE/\cO
\end{equation*}
the quotient bundle.
The embedding $\tcE \hookrightarrow W \otimes \cO_{\bX^+}$ induces an embedding of vector bundles 
\begin{equation*}
\cE \xhookrightarrow{\hspace{2em}} \barW \otimes \cO
\end{equation*} 
and a morphism
\begin{equation}
\label{eq:phiplus}
\phi_+ \colon \PP_{\bX^+}(\cE) \xlongrightarrow{\hspace{2em}}\PP(\barW).
\end{equation} 
We denote by $\barbH$ the hyperplane class of~$\PP(\barW)$;
its pullback to~$\PP_{\bX^+}(\cE)$ is then a relative hyperplane class.
We consider the blowup of~$\PP_{\bX^+}(\tcE)$ along~$s_{x_0}(\bX^+)$ and denote by~$\bE_\PP$ its exceptional divisor.

\begin{lemma}
\label{lemma:pp-blowup}
There is an isomorphism of $\bX^+$-schemes
\begin{equation*}
\Bl_{s_{x_0}(\bX^+)}(\PP_{\bX^+}(\tcE)) \cong 
\PP_{\PP_{\bX^+}(\cE)}(\cO \oplus \cO(-\barbH)).
\end{equation*}
Moreover, $\bH$ is a relative hyperplane class for~$\PP_{\PP_{\bX^+}(\cE)}(\cO \oplus \cO(-\barbH))$ and there is a linear equivalence
\begin{equation*}
\bE_\PP \sim \bH - \barbH.
\end{equation*}
Finally, $\bE_\PP$ is equal to the section of the projection
\begin{equation*}
\pi_\cE \colon \PP_{\PP_{\bX^+}(\cE)}(\cO \oplus \cO(-\barbH)) \xlongrightarrow{\hspace{2em}} \PP_{\bX^+}(\cE)
\end{equation*}
that corresponds to the embedding~$\cO \hookrightarrow \cO \oplus \cO(-\barbH)$.
\end{lemma}
\begin{proof}
This is just a relative version of the isomorphism of the blowup of a projective space at a point
and the $\PP^1$-bundle over the projective space of dimension by one less.
\end{proof}

We define 
\begin{equation*}
\hbX := \Bl_{s_{x_0}(\bX^+)}(\cQ)
\end{equation*}
to be the blowup of~$\cQ$ 
along the image of the section~\eqref{eq:section-cq}.

\begin{lemma}
\label{lemma:varpi-varpiplus}
There exist morphisms $\pi$, $\hat\iota$, and~$\pi_+$ that make the following diagram commutative
\begin{equation}
\label{diagram:pre}
\vcenter{\xymatrix@C=4em{
&
\hbX \ar[d]^{\hat\sigma} \ar@{..>}[dl]_\pi \ar@{..>}[r]^-{\hat\iota} \ar@{..>}[drr]_(.25){\pi_+} &
\Bl_{s_{x_0}(\bX^+)}(\PP_{\bX^+}(\tcE)) \ar[d] \ar[dr]^{\pi_\cE}
\\
\Bl_{x_0}(\bX) \ar[d]_\sigma &
\cQ \ar[drr]_(.25){q_\cQ} \ar[dl]_{p_\cQ} \ar@{^{(}->}[r]^{\iota} &
\PP_{\bX^+}(\tcE) \ar[dr] &
\PP_{\bX^+}(\cE) \ar[d]^{\sigma_+} 
\\
\bX &&&
\bX^+,
}}
\end{equation}
where $\sigma$ and~$\hat\sigma$ are the blowup maps and~$\sigma_+$ is the projective bundle morphism.
Moreover, $\pi$ is birational.
\end{lemma}
\begin{proof}

Since the scheme-theoretic preimage of the point $x_0$ via the map $p_\cQ$ coincides with $s_{x_0}(\bX^+)$,
its scheme-theoretic preimage in $\hbX$ is a Cartier divisor, 
hence there is a map $\pi \colon \hbX \to \Bl_{x_0}(\bX)$ making the left square commutative.
It is birational, because so are the morphisms~$\sigma$, $\hat\sigma$, and~$p_\cQ$
(for~$p_\cQ$ this is assumption~\ref{ass:p2}).

Similarly, there exists a morphism $\hat\iota \colon \hbX \xrightarrow{\hspace{1em}} \Bl_{s_{x_0}(\bX^+)}(\PP_{\bX^+}(\tcE))$ in the diagram
such that the central square commutes.
Finally, $\pi_+$ can be defined as the composition~$\pi_\cE \circ \hat\iota$.
\end{proof}

We need to identify the maps $\pi$ and $\pi_+$ as blowups.
We start with~$\pi_+$.

\begin{lemma}
\label{lemma:varpiplus}
There is an embedding~$\rF_1(\cQ/\bX^+,x_0) \hookrightarrow \PP_{\bX^+}(\cE)$
such that the morphism~$\pi_+$ is the blowup of\/~$\rF_1(\cQ/\bX^+,x_0)$,
and the exceptional divisor\/~$\hbD$ of\/~$\pi_+$ is isomorphic to~$\cL(\cQ/\bX^+,x_0)$.
In particular, $\hbX$ and~$\hbD$ are smooth.
\end{lemma}
\begin{proof}
There is a natural isomorphism of the relative Hilbert scheme of lines of~$\PP_{\bX^+}(\tcE)$
\begin{equation*}
\rF_1(\PP_{\bX^+}(\tcE)/\bX^+,x_0) \cong \PP_{\bX^+}(\cE),
\end{equation*}
such that the universal line $\cL(\PP_{\bX^+}(\tcE)/\bX^+,x_0)$ is isomorphic 
to~$\PP_{\PP_{\bX^+}(\cE)}(\cO \oplus \cO(-\barbH))$;
this in particular defines a natural embedding~$\rF_1(\cQ/\bX^+,x_0) \hookrightarrow \PP_{\bX^+}(\cE)$.

Since~$\iota(\cQ)$ is a divisor of type~$2\bH + D$ (see~\eqref{eq:cq-class})
it follows that~$\rF_1(\cQ/\bX^+,x_0)$ is the zero locus of a section
of the pushforward to~$\PP_{\bX^+}(\cE)$ of the line bundle~$\cO(2\bH + D - \bE_\PP)$ 
on~\mbox{$\PP_{\PP_{\bX^+}(\cE)}(\cO \oplus \cO(-\barbH))$},
On the other hand, $\hbX$ is the strict transform of~$\cQ$ in~$\Bl_{s_{x_0}(\bX^+)}(\PP_{\bX^+}(\tcE))$, 
hence it is the zero locus of the same section of~$\cO(2\bH + D - \bE_\PP)$.
By Lemma~\ref{lemma:pp-blowup} we have 
\begin{equation*}
\cO(2\bH + D - \bE_\PP) \cong \cO(\bH + \barbH + D),
\end{equation*}
so this is a relative hyperplane class for 
the $\PP^1$-bundle~$\PP_{\PP_{\bX^+}(\cE)}(\cO \oplus \cO(-\barbH))$.
Applying~\cite[Lemma~2.1]{K16} we conclude that
\begin{equation*}
\hbX \cong \Bl_{\rF_1(\cQ/\bX^+,x_0)}(\PP_{\bX^+}(\cE))
\end{equation*}
since $\codim_{\PP_{\bX^+}(\cE)}(\rF_1(\cQ/\bX^+,x_0)) = \dim(\bX) - \dim(\rF_1(\cQ/\bX^+,x_0)) = 2$
by assumption~\ref{ass:f1}.

This argument also proves that the exceptional divisor~$\hbD$ of the blowup coincides with the universal line~$\cL(\cQ/\bX^+,x_0)$.
The smoothness of the varieties~$\hbX$ and~$\hbD$ follows from the smoothness of~$\bX^+$ and~$\rF_1(\cQ/\bX^+,x_0)$, assumptions~\ref{ass:f2k} and~\ref{ass:f1}.
\end{proof}

Now we can describe the morphism~$\pi$.

\begin{lemma}
\label{lemma:varpi}
There is a natural embedding $\cL(\bX,x_0) \hookrightarrow \Bl_{x_0}(\bX)$ such that
the morphism~$\pi$ is the blowup with center~$\cL(\bX,x_0)$.
Moreover,
\begin{equation}
\label{eq:be-cap-cl}
\bE \cap \cL(\bX,x_0) \cong \rF_1(\bX,x_0),
\end{equation}
the intersection is transverse, and $\pi^{-1}(\bE) \cong \Bl_{\rF_1(\bX,x_0)}(\bE)$.
\end{lemma}
\begin{proof}
The scheme-theoretic preimage of the point~$x_0$ along the map $p_\cL \colon \cL(\bX,x_0) \to \bX$ 
is equal to the image of the section $s_{x_0} \colon \rF_1(\bX,x_0) \to \cL(\bX,x_0)$;
in particular, it is a Cartier divisor on~$\cL(\bX,x_0)$. 
Therefore, the morphism $p_\cL$ lifts to a morphism
\begin{equation*}
\cL(\bX,x_0) \xrightarrow{\hspace{1em}} \Bl_{x_0}(\bX),
\end{equation*}
which is clearly a closed embedding.
Moreover, the scheme-theoretic preimage of~$\bE$ under this embedding 
is equal to the image of the section~$s_{x_0}$, which proves~\eqref{eq:be-cap-cl}.
Since the intersection is smooth (by assumption~\ref{ass:f1}), it is transverse.

Consider the birational morphism $\pi \colon \hbX \to \Bl_{x_0}(\bX)$.
Its source is smooth by Lemma~\ref{lemma:varpiplus} and its target is smooth by the assumptions.
Furthermore, from the diagram~\eqref{diagram:pre} we deduce that
\begin{equation*}
\sigma(\pi(\hbD)) = 
p_\cQ(\hat\sigma(\hbD)) = 
p_\cQ(\hat\sigma(\cL(\cQ/\bX^+,x_0))) = 
p_\cQ(\tilde{p}_\cL(\cL(\cQ/\bX^+,x_0))) = 
p_\cL(\cL(\bX,x_0))
\end{equation*}
(the last equality uses Lemma~\ref{lemma:lines}).
Since~$\hbD$ is irreducible, it follows that $\pi(\hbD) = \cL(\bX,x_0)$.
Note that $\codim(\cL(\bX,x_0)) = \dim(\bX) - \dim(\rF_1(\bX,x_0)) - 1 > 1$ (again by assumption~\ref{ass:f1}).
Finally, we have 
\begin{equation*}
\uprho(\hbX) = \uprho(\bX^+) + 2 = \uprho(\bX) + 2 = \uprho(\Bl_{x_0}(\bX)) + 1,
\end{equation*}
hence by~\cite[Lemma~2.5]{K18} we conclude that $\pi$ is the blowup with center $\cL(\bX,x_0)$.

Since the divisor~$\bE$ intersects the blowup center transversely along~$\rF_1(\bX,x_0)$, 
its preimage is equal to its strict transform and isomorphic to the blowup with center~$\rF_1(\bX,x_0)$.
\end{proof}

Finally, we can combine the above results and finish the proof of the theorem.

\begin{proof}[Proof of Theorem~\xref{theorem:qk}]
The morphisms~$\pi$ and~$\pi_+$ are constructed in Lemma~\ref{lemma:varpi-varpiplus}
and they are proved to be the blowups in Lemma~\ref{lemma:varpiplus} and~\ref{lemma:varpi}, respectively.
The description of the divisor~$\hbD$ as the fiber product follows 
from a combination of Lemma~\ref{lemma:varpiplus} with~\eqref{eq:lines}.

The morphism~$\phi$ is induced by the linear projection from~$x_0$, and the morphism~$\phi_+$ is defined in~\eqref{eq:phiplus}.
We have $\phi \circ \pi = \phi_+ \circ \pi_+$, because the two maps can be rewritten as the compositions
\begin{equation*}
\xymatrix@R=.5ex@C=4em{
&& \bX \ar@{^{(}->}[dr]
\\
\hbX \ar[r]^{\hat\sigma} &
\cQ \ar[ur]^{p_\cQ} \ar@{^{(}->}[dr]_(.4){\iota} &&
\PP(W) \ar@{-->}[r] &
\PP(\barW),
\\
&& \PP_{\bX^+}(\tcE) \ar[ur] \ar@{-->}[r] &
\PP_{\bX^+}(\cE) \ar[ur]
}
\end{equation*}
and the squares commute.

The equality $\phi \circ \pi = \phi_+ \circ \pi_+$ implies that $\phi(\Bl_{x_0}(\bX)) = \phi_+(\PP_{\bX^+}(\cE))$ 
and~\mbox{$\pi_+ \circ \pi^{-1} = \phi_+^{-1} \circ \phi$},
thus defining the subvariety $\barbX \subset \PP(\barW)$ and the birational map~$\psi$.
Since the morphisms~$\pi$ and~$\pi_+$ are contractions with the same exceptional divisor, $\psi$ is small.
Assumption~\ref{ass:f1} implies that~$\codim_{\Bl_{x_0}(\bX)}(\cL(\bX,x_0)) \ge \codim_{\PP_{\bX^+}(\cE)}(\rF_1(\cQ/\bX^+,x_0))$, 
hence~$\psi$ is a flop or antiflip.

It remains to construct the subsheaf~$\cE_0 \subset \cE$, to prove~\eqref{eq:hbe}, \eqref{eq:iso-bx-open}, and~\eqref{eq:iso-be-open},
and to describe the restrictions of $\pi$, $\pi_+$, $\phi$, and $\phi_+$ to various strata.

Consider the diagram~\eqref{diagram:pre}.
By construction of~$\pi$, the preimage of~$\bE$ in~$\hbX$ is the exceptional divisor of~$\hat\sigma$,
which is also equal $\hbX \cap \bE_\PP$.
The divisor~$\bE_\PP$ by Lemma~\ref{lemma:pp-blowup} is equal to the section of the morphism~$\pi_\cE$
corresponding to the embedding~$\cO \hookrightarrow \cO \oplus \cO(-\barbH)$; 
in particular~$\bE_\PP$ is isomorphic to~$\PP_{\bX^+}(\cE)$ via~$\pi_\cE$,
and the restriction of the hyperplane class~$\bH$ to~$\bE_\PP$ is trivial.

As it is explained in Lemma~\ref{lemma:varpiplus}, 
$\hbX$ is the divisor of type~\mbox{$\bH + \barbH + D$} in~$\PP_{\PP_{\bX^+}(\cE)}(\cO \oplus \cO(-\barbH))$,
hence $\hbX \cap \bE_\PP$ is equal to the zero locus of a section of the line bundle $\cO(\barbH + D)$ on~$\bE_\PP = \PP_{\bX^+}(\cE)$.
This section corresponds to a morphism
\begin{equation*}
\cE \xlongrightarrow{\hspace{2em}} \cO(D)
\end{equation*}
on~$\bX^+$.
If we define~$\cE_0$ as its kernel, then $\hbX \cap \bE_\PP = \PP_{\bX^+}(\cE_0)$, 
and the morphism~$\pi_\cE$ embeds it naturally into~$\PP_{\bX^+}(\cE)$.
This proves~\eqref{eq:hbe}.

The maps~$\pi$ and~$\pi_+$ in~\eqref{eq:iso-bx-open} are isomorphisms by Lemma~\ref{lemma:varpi} and~\ref{lemma:varpiplus}, respectively.
Furthermore, the map~$\phi$ is an isomorphism because it is induced by the linear projection
and the map~$\phi_+$ is an isomorphism because $\phi_+ = \phi \circ \pi \circ \pi_+^{-1}$.

The restriction of~$\pi_+$ to~$\hbD$ coincides with the projection to the second factor by Lemma~\ref{lemma:varpiplus}.
Similarly, by Lemma~\ref{lemma:varpi} the restriction of~$\pi$ to~$\hbD$ 
coincides with the map \mbox{$\cL(\cQ/\bX^+,x_0) \to \cL(\bX,x_0)$} 
induced by the morphism~\eqref{eq:f1-morphism}, hence it is the projection to the first factor.
On the other hand, since~$\pi$ is a smooth blowup, this morphism is a projective bundle,
hence so is the morphism~\eqref{eq:f1-morphism}.
The restriction of~$\phi$ to~$\cL(\bX,x_0)$ is the natural $\PP^1$-bundle, 
because~$\phi$ is induced by the linear projection from~$x_0$,
and the restriction of~$\phi_+$ to~$\rF_1(\cQ/\bX^+,x_0)$ is the map~\eqref{eq:f1-morphism}, 
because it is the only map such that $\phi_+ \circ \pi_+ = \phi \circ \pi$ on~$\hbD$.

Finally, the isomorphism~\eqref{eq:iso-be-open} follows from~\eqref{eq:hbe} and~\eqref{eq:iso-bx-open}.
\end{proof}

\subsection{Grassmannians of lines}
\label{subsection:gr2v}

In this section we show that the assumptions of Theorem~\ref{theorem:qk} are satisfied for Grassmannians~$\bX = \Gr(2,V)$. 
Note that when~$\dim(V) = 6$ the Grassmannian~$\Gr(2,V)$ is the maximal Mukai variety~$\bX_{14}$ of genus~$g = 8$.
In this case we show that diagram~\eqref{diagram:typical} exists with $\bX^+ = \Gr(2,4)$.

Assume~$\dim(V) \ge 5$.
We consider the Pl\"ucker embedding $\bX := \Gr(2,V) \hookrightarrow \PP(\wedge^2V)$, so that~$\bH$ is the Pl\"ucker polarization.
Let $U_0 \subset V$ be a $2$-dimensional subspace and let $x_0 \in \bX$ be the corresponding point.
We denote the quotient space by 
\begin{equation*}
V^+ := V/U_0.
\end{equation*}
We further denote 
\begin{equation*}
\bX^+ := \Gr(2,V^+),
\qquad 
\cQ := \Fl(2,4;V) \times_{\Gr(4,V)} \bX^+,
\end{equation*}
where the embedding $\bX^+ \hookrightarrow \Gr(4,V)$ takes a subspace $U^+ \subset V^+$ 
to its preimage under the projection $V \to V^+$ along $U_0$.
Note that $\cQ \to \bX^+$ is a $\Gr(2,4)$-fibration, so its fibers are smooth $4$-dimensional quadrics on~$\bX$ passing through~$x_0$.

\begin{lemma}
\label{lemma:f1-gr2m}
There is a commutative diagram
\begin{equation*}
\xymatrix{
\rF_1(\cQ/\bX^+,{x_0}) \ar[d] \ar@{=}[r] &
\PP(U_0) \times \Fl(1,2;V^+) \ar[d]
\\
\rF_1(\bX,x_0) \ar@{=}[r] & 
\PP(U_0) \times \PP(V^+),
}
\end{equation*}
where the map on the left is the map~\eqref{eq:f1-morphism} and the map on the right is induced by the natural map $\Fl(1,2;V^+) \to \Gr(2,V^+)$.
In particular, $\rF_1(\bX,x_0)$ is smooth, $\rF_1(\cQ/\bX^+,{x_0})$ is smooth over~$\rF_1(\bX,x_0)$, connected,
and 
\[
\dim(\rF_1(\cQ/\bX^+,{x_0})) = 2\dim(V) - 6 = \dim(\bX) - 2.
\]
\end{lemma}
\begin{proof}
Recall that any line on~$\Gr(2,V)$ has the form
\begin{equation}
\label{eq:line-v1-v3}
\ell(V_1,V_3) = \{ [U] \mid V_1 \subset U \subset V_3 \},
\end{equation} 
where $V_1 \subset V_3 \subset V$ is a flag with $\dim(V_i) = i$.
Clearly, $[U_0] \in \ell(V_1,V_3)$ means $V_1 \subset U_0 \subset V_3$; 
such lines are determined by the points~$[V_1] \in \PP(U_0)$ and $[V_3/U_0] \in \PP(V/U_0) = \PP(V^+)$;
this proves the bottom equality.

The line $\ell(V_1,V_3)$ lies on the quadric $\cQ_{[U^+]}$ if and only if $V_3/U_0 \subset U^+$.
Therefore, the fiber of the map~\eqref{eq:f1-morphism} over $[\ell(V_1,V_3)]$
parameterizes all $U^+ \subset V^+$ such that $V_3/U_0 \subset U^+$.
This is the space $\PP(V/V_3)$, i.e., the fiber of the projection $\Fl(1,2;V^+) \to \Gr(2,V^+)$.
This proves the second part of the lemma.
\end{proof}

\begin{proposition}
\label{proposition:gr2v}
The assumptions of Theorem~\xref{theorem:qk} are satisfied.
Moreover, 
\begin{equation*}
\cE_0 \cong U_0 \otimes \cU^+,
\qquad 
\cE \cong \cE_0 \oplus \wedge^2\cU^+,
\end{equation*}
where $\cU^+$ is the tautological bundle on~$\Gr(2,V^+)$
and 
\[
\sigma_+ \circ \psi \colon \Bl_{x_0}(\Gr(2,V)) \xdashrightarrow{\hspace{2em}} \Gr(2,V^+)
\]
is given by the linear system $|\bH - 2\bE|$.
\end{proposition}

\begin{proof}
Assumption~\ref{ass:f2k} of Theorem~\ref{theorem:qk} is evident.

The fiber of the map $p_\cQ \colon \cQ \to \bX$ over the point $[U] \in \bX = \Gr(2,V)$ corresponding to a subspace~\mbox{$U \subset V$}
parameterizes $4$-dimensional subspaces in $V$ that contain both $U$ and $U_0$.
Generically, this is just one point $[U_0 + U]$, so $p_\cQ$ is birational,
hence assumption~\ref{ass:p2} is satisfied.

Assumption~\ref{ass:qk} is also clear and assumption~\ref{ass:f1} is proved in Lemma~\ref{lemma:f1-gr2m}.

Furthermore, we have
\begin{equation*}
\tcE = \wedge^2(U_0 \otimes \cO \oplus \cU^+) = (\wedge^2U_0 \otimes \cO) \oplus (U_0 \otimes \cU^+) \oplus \wedge^2\cU^+.
\end{equation*}
The section $s_{x_0}$ corresponds to the first summand
and the subbundle~$\cE_0$ is the second summand.

The fact that the map $\sigma_+ \circ \psi$ is given by the linear system $|\bH - 2\bE|$ is evident 
(and in the case when $\dim(V) > 5$ it also follows from Lemma~\ref{lemma:linear-system}).
\end{proof}

\subsection{Orthogonal Grassmannian}
\label{subsection:ogr510}

In this section we show that the assumptions of Theorem~\ref{theorem:qk} are satisfied 
for~$\bX= \OGr_+(5,10)$,
the maximal Mukai variety~$\bX_{12}$ of genus~7.
We refer to~\cite{K18} for basic facts about its geometry.
In this case we show that diagram~\eqref{diagram:typical} exists with $\bX^+ = \PP^4$.
We assume here that the base field is algebraically closed.

We denote by $V$ a 10-dimensional vector space endowed with a non-degenerate quadratic form,
consider the half-spinor embedding $\OGr_+(5,V) \to \PP(W) \cong \PP^{15}$, 
and denote by $\bH$ the corresponding polarization.
Let $U_0 \subset V$ be a $5$-dimensional isotropic subspace and let $x_0 \in \bX$ be the corresponding point.
We denote 
\begin{equation*}
\bX^+ := \PP(U_0),
\qquad 
\cQ := \OFl_+(1,5;V) \times_{\PP(V)} \bX^+,
\end{equation*}
where we consider the natural embedding $\bX^+ = \PP(U_0) \hookrightarrow \PP(V)$.
Note that $\cQ \to \bX^+$ is a family of smooth $6$-dimensional quadrics on~$\bX$, see~\cite[\S3.1]{K18}, passing through~$x_0$.

\begin{lemma}
\label{lemma:f1-ogr}
There is a commutative diagram
\begin{equation*}
\xymatrix{
\rF_1(\cQ/\bX^+,{x_0}) \ar[d] \ar@{=}[r] &
\Fl(1,3;U_0) \ar[d]
\\
\rF_1(\bX,x_0) \ar@{=}[r] & 
\Gr(3,U_0),
}
\end{equation*}
where the map on the left is the morphism~\eqref{eq:f1-morphism} and the map on the right is the natural projection.
In particular, $\rF_1(\bX,x_0)$ is smooth, $\rF_1(\cQ/\bX^+,{x_0})$ is smooth over~$\rF_1(\bX,x_0)$, connected, and
\begin{equation*}
\dim(\rF_1(\cQ/\bX^+,x_0)) = 8 = \dim(\bX) - 2.
\end{equation*}
\end{lemma}
\begin{proof}
By~\cite[Theorem~3.2]{K18} the Hilbert scheme of lines on~$\OGr_+(5,V)$ is~$\OGr(3,V)$,
and the line $\ell(V_3)$ corresponding to an isotropic subspace $V_3 \subset V$ 
contains the point~$[U_0]$ if and only if~\mbox{$V_3 \subset U_0$};
this proves the bottom equality.

The fiber of the map~\eqref{eq:f1-morphism} over $[\ell(V_3)]$ parameterizes all $V_1 \subset V$ such that $V_1 \subset V_3$;
this proves the second part of the lemma.
\end{proof}

\begin{proposition}
\label{proposition:ogr510}
The assumptions of Theorem~\xref{theorem:qk} are satisfied.
Moreover, 
\begin{equation*}
\cE_0 \cong \Omega_{\PP(U_0)}^2(2),
\qquad 
\cE \cong \cE_0 \oplus \cO_{\PP(U_0)}(-1),
\end{equation*}
and the map $\sigma_+ \circ \psi \colon \Bl_{x_0}(\OGr_+(5,V)) \dashrightarrow \PP(U_0)$ is given by the linear system $|\bH - 2\bE|$.
\end{proposition}

\begin{proof}
Assumption~\ref{ass:f2k} of Theorem~\ref{theorem:qk} is evident.

The fiber of the map $p_\cQ \colon \cQ \to \bX$ over the point $[U] \in \bX = \OGr_+(5,V)$ 
corresponding to an isotropic subspace~\mbox{$U \subset V$} is equal to~$\PP(U_0 \cap U)$.
Since $[U]$ is contained in the same connected component of $\OGr(5,V)$ as $[U_0]$, we have 
\begin{equation*}
\dim(U \cap U_0) \equiv \dim(U_0) \bmod 2,
\end{equation*}
see~\cite[\S2.1]{K18}.
Therefore, for $[U]$ general in~$\OGr_+(5,V)$ we have~$\dim(U \cap U_0) = 1$,
hence the general fiber of $p_\cQ$ is a point, $p_\cQ$ is birational, 
hence assumption~\ref{ass:p2} is satisfied.

Assumption~\ref{ass:qk} follows from~\cite[Corollary~4.7]{K18} and assumption~\ref{ass:f1} is proved in Lemma~\ref{lemma:f1-ogr}.

Furthermore, by~\cite[(3.3)]{K18} the bundle $\tcE$ is the restriction to $\PP(U_0)$ of the spinor bundle on the quadric $\OGr(1,V)$, 
hence by~\cite[Theorem~2.6]{Ottaviani} we have
\begin{equation*}
\tcE \cong \cO_{\PP(U_0)} \oplus \Omega_{\PP(U_0)}^2(2) \oplus \Omega_{\PP(U_0)}^4(4).
\end{equation*}
The section $s_{x_0}$ corresponds to the first summand, 
and the subbundle~$\cE_0$ is the second summand.

The fact that the map $\sigma_+ \circ \psi$ is given by the linear system $|\bH - 2\bE|$ follows from Lemma~\ref{lemma:linear-system}.
\end{proof}

\subsection{Grassmannian of the group~$\rG_2$}
\label{subsection:g2gr}

In this section we show that the assumptions of Theorem~\ref{theorem:qk} are satisfied for~$\bX = \GTGr(2,7)$,
the maximal Mukai variety~$\bX_{18}$ of genus~10.
In this case we show that diagram~\eqref{diagram:typical} exists with $\bX^+$ a quintic del Pezzo fourfold.
We assume again that the base field is algebraically closed.

We denote by $V$ a $7$-dimensional vector space endowed with a general $3$-form $\lambda \in \wedge^3V^\vee$.
Then
\begin{equation*}
\bX = \GTGr(2,V) \subset \Gr(2,V)
\end{equation*}
is the locus of subspaces $U \subset V$ annihilated by~$\lambda$.
We have $\dim(\bX) = 5$, $\Pic(\bX) = \ZZ \bH$, where~$\bH$ is the hyperplane class,
and $\bX$ is a homogeneous space for the action of the stabilizer of~$\lambda$ in~$\GL(V)$,
which is the simple algebraic group~$\rG_2$.

The composition of the embedding of~$\bX$ with the Pl\"ucker embedding of~$\Gr(2,V)$ 
factors through the projectivization of the adjoint ($14$-dimensional) representation of~$\rG_2$
\begin{equation}
\label{eq:w-g2}
W = \Ker(\wedge^2V \xrightarrow{\quad \lambda\quad} V^\vee) \subset \wedge^2V,
\end{equation}
and, moreover, $\bX = \Gr(2,V) \cap \PP(W)$, although the intersection is highly non-transverse.

Let $U_0 \subset V$ be a $2$-dimensional subspace annihilated by~$\lambda$ 
and let $x_0 \in \bX$ be the corresponding point.
If $u'_0,u''_0$ is a basis of $U_0$, the $2$-forms
\begin{equation*}
\lambda' = \lambda(u'_0,-,-)
\qquad\text{and}\qquad 
\lambda'' = \lambda(u''_0,-,-)
\end{equation*}
annihilate $U_0$, hence induce $2$-forms on the $5$-dimensional quotient space 
\begin{equation*}
V^+ = V/U_0,
\end{equation*}
so we can write $\lambda',\lambda'' \in \wedge^2(V^+)^\vee$.
We define
\begin{equation*}
\bX^+ \subset \Gr(2,V^+) = \Gr(2,5)
\end{equation*}
as the zero locus of $\lambda'$ and $\lambda''$, considered as sections of $\cO_{\Gr(2,V^+)}(1)$.
Note that $\Gr(2,V^+)$ can be identified with the subvariety of $\Gr(4,V)$ parameterizing subspaces which contain~$U_0$.

\begin{lemma}
\label{lemma:xq-g2}
The variety $\bX^+$ is a smooth quintic del Pezzo fourfold, and
\begin{equation*}
\cQ := \Fl(2,4;V) \times_{\Gr(2,V) \times \Gr(4,V)} (\bX \times \bX^+)
\end{equation*}
is a flat family of conics on~$\bX$ passing through~$x_0$ and parameterized by~$\bX^+$.
\end{lemma}
\begin{proof}
For the first statement we need to check that 
no linear combination of $\lambda'$ and $\lambda''$ has rank~$2$ or less (see~\cite[Proposition~2.24]{DK18}).
But it is a classical fact, that $\lambda(v,-,-)$ for any $0 \ne v \in V$ has rank at least~4 (see, e.g.,~\cite[Lemma~3.5]{K16}),
so the first statement follows.

For the second statement we note that the fiber of $\cQ \to \bX^+$ over the point $[U^+]$ 
corresponding to a two-subspace $U^+ \subset V^+$ is the intersection $\Gr(2,U_0 \oplus U^+) \cap \PP(W)$
(where we implicitly have chosen a splitting of $V$ into the sum $U_0 \oplus V^+$ and thus consider $U^+$ as a subspace of $V$).
By~\eqref{eq:w-g2} we have
\begin{equation}
\label{eq:intersection}
\wedge^2(U_0 \oplus U^+) \cap W = 
\Ker\bigl(\wedge^2(U_0 \oplus U^+) \xhookrightarrow{\hspace{2em}} \wedge^2V \xrightarrow{\quad \lambda\quad} V^\vee\bigr).
\end{equation}
On the other hand, $\lambda$ restricts trivially to $U_0 \oplus U^+$; indeed, if $u_+',u_+''$ is a basis of $U^+$ then
\begin{equation*}
\lambda(u'_0,u''_0,u'_+) = \lambda(u'_0,u''_0,u''_+) =0 
\end{equation*}
because $\lambda$ annihilates~$U_0$, and
\begin{equation*}
\lambda(u'_0,u'_+,u''_+) = \lambda(u''_0,u'_+,u''_+) = 0
\end{equation*}
because $\lambda'$ and $\lambda''$ vanish on $U^+$.
Therefore, the map $\wedge^2(U_0 \oplus U^+) \to V^\vee$ in the right side of~\eqref{eq:intersection}
factors through the $3$-dimensional subspace $(U_0 \oplus U^+)^\perp \subset V^\vee$, 
hence 
\begin{equation*}
\dim(\wedge^2(U_0 \oplus U^+) \cap W) \ge 3.
\end{equation*}
This means that the fiber $\cQ_{[U^+]} = \Gr(2,U_0 \oplus U^+) \cap \PP(\wedge^2(U_0 \oplus U^+) \cap W)$ of~$q_\cQ$
is either a conic, or a plane, or a quadric of dimension at least~$2$.
But $\bX$ contains neither planes, nor quadric surfaces~\cite[Lemma~3]{KR13}, hence~$\cQ_{[U^+]}$ is a conic for each~$[U^+]$, 
so $\cQ \to \bX^+$ is a flat conic bundle.
Each of the conics~$\cQ_{[U^+]}$ passes through~$x_0$ by construction.
\end{proof}

\begin{lemma}
\label{lemma:f1-g2}
We have $\rF_1(\bX,x_0) \cong \PP(U_0) \cong \PP^1$ and the morphism~\eqref{eq:f1-morphism}
is a $\PP^2$-bundle.
In particular, the scheme $\rF_1(\bX,x_0)$ is smooth, $\rF_1(\cQ/\bX^+,{x_0})$ is smooth over~$\rF_1(\bX,x_0)$, connected,
and
\begin{equation*}
\dim(\rF_1(\cQ/\bX^+,x_0)) = 3 = \dim(\bX) - 2.
\end{equation*}
\end{lemma}
\begin{proof}
A line~$\ell(V_1,V_3) \subset \Gr(2,V)$ defined in~\eqref{eq:line-v1-v3} lies on~$\bX$ 
if an only if $V_3$ is contained in the kernel of the $2$-form~$\lambda(V_1,-,-)$.
Such~$V_3$ exists only if the rank of~$\lambda(V_1,-,-)$ is at most~4.
Points $[V_1] \in \PP(V)$ with this property are parameterized by a smooth quadric~$\bQ \subset \PP(V)$ (see, e.g., \cite[Lemma~3.5]{K16}).
Moreover, for any such~$[V_1]$ the rank of~$\lambda(V_1,-,-)$ is exactly~4, so~$V_3$ is determined by~$V_1$.
This proves that $\rF_1(\bX) \cong \bQ$.
The line $\ell(V_1,V_3)$ contains~$[U_0]$ if and only if~$V_1 \subset U_0$,
hence $\rF_1(\bX,x_0) = \PP(U_0)$.

Furthermore, the fiber of the morphism $\rF_1(\cQ/\bX^+,{x_0}) \to \rF_1(\bX,x_0)$ over a point $[V_1] \in \PP(U_0)$
is the locus of all $[U^+] \in \bX^+$ such that $V_3 := \Ker(\lambda(V_1,-,-))$ is contained in $U_0 \oplus U^+$.
Clearly, we can take $U^+$ to be any subspace of $V^+$ containing the line corresponding to $V_1^+ := V_3/U_0 \subset V^+$ 
and isotropic for the $2$-form $\lambda(u_0,-,-)$ for any $u_0 \in U_0 \setminus V_1$.
Such subspaces are parameterized by~$\PP(\Ker(\lambda(u_0,V_3,-))/V_3) \cong \PP^2$.
This proves that the morphism is a $\PP^2$-bundle.
\end{proof}

Recall~\cite{Todd1930}, \cite[3.3]{Debarre-Iliev-Manivel2012} 
that the Hilbert scheme of planes on the quintic del Pezzo fourfold~$\bX^+$ has two components.
The first component is isomorphic to~$\PP^1$.
In fact, it can be identified with $\rF_1(\bX,x_0)$ and using the argument of Lemma~\ref{lemma:f1-g2} one can check that
the universal plane on~$\bX^+$ coincides with the $\PP^2$-bundle $\rF_1(\cQ/\bX^+,{x_0}) \to \rF_1(\bX,x_0)$.
Thus, the threefold~$\rF_1(\cQ/\bX^+,{x_0})$ is the normalization of the divisor on $\bX_+$ swept by planes of the first type.

The second component of the Hilbert scheme of planes is a single point, and corresponds to a special plane $\Pi \subset \bX^+$.
In fact, $\Pi = \Gr(2,V_3^+) \subset \Gr(2,V^+)$, where $V_3^+ \subset V^+$ 
is the unique $3$-dimensional subspace 
isotropic for all two-forms in the pencil generated by~$\lambda'$ and~$\lambda''$.

\begin{proposition}
\label{proposition:g2gr}
The assumptions of Theorem~\xref{theorem:qk} are satisfied.
Moreover, 
\begin{equation*}
\cE_0 \cong \cO_{\bX^+}(-1)
\end{equation*}
and the map $\sigma_+ \circ \psi \colon \Bl_{x_0}(\GTGr(2,V)) \dashrightarrow \bX^+$ is given by the linear system $|\bH - 2\bE|$.
\end{proposition}

\begin{proof}
Assumption~\ref{ass:f2k} of Theorem~\ref{theorem:qk} is evident.

By Lemma~\ref{lemma:xq-g2} the fiber of the map $p_\cQ \colon \cQ \to \bX$ over a point $[U] \in \bX = \GTGr(2,V)$ 
corresponding to a subspace~\mbox{$U \subset V$} parametrizes all $4$-dimensional subspaces of~$V$ that contain~$U_0$ and~$U$
and to which~$\lambda$ restricts trivially.
If $U \cap U_0 = 0$, the subspace $U_0 \oplus U$ is unique with this property (see Lemma~\ref{lemma:xq-g2}).
Therefore, the map~$p_\cQ$ is birational, 
hence assumption~\ref{ass:p2} is satisfied.

Assumption~\ref{ass:qk} follows from~\cite[Lemma~3]{KR13} and assumption~\ref{ass:f1} is proved in Lemma~\ref{lemma:f1-g2}.

Furthermore, the proof of Lemma~\ref{lemma:xq-g2} shows that the vector bundle $\tcE$ is defined by the exact sequence
\begin{equation*}
0 \xrightarrow{\hspace{1em}} \tcE \xrightarrow{\hspace{1em}} \wedge^2(U_0 \otimes \cO \oplus \cU^+) \xrightarrow{\quad \lambda\quad } (U_0 \otimes \cO \oplus \cU^+)^\perp \xrightarrow{\hspace{1em}} 0,
\end{equation*}
where $\cU^+$ is the tautological bundle on~$\bX^+$.
Moreover, the section $s_{x_0}$ corresponds to the natural embedding of $\wedge^2U_0 \otimes \cO$ into $\wedge^2(U_0 \otimes \cO \oplus \cU^+)$.
Therefore, the bundle~$\cE$ is defined by the exact sequence 
\begin{equation*}
0 \xrightarrow{\hspace{1em}} \cE \xrightarrow{\hspace{1em}} U_0 \otimes \cU^+ \oplus \wedge^2\cU^+ \xrightarrow{\quad \lambda\quad} (U_0 \otimes \cO \oplus \cU^+)^\perp \xrightarrow{\hspace{1em}} 0.
\end{equation*}
If the first component 
\begin{equation*}
\lambda_1 \colon U_0 \otimes \cU^+ \xrightarrow{\quad  \quad} (U_0 \otimes \cO \oplus \cU^+)^\perp
\end{equation*}
of the above map
is not surjective at point $[U^+]$ then there is a line $V^+_1 \subset V^+ \setminus U^+$ such that
\begin{equation*}
\lambda'(U^+,V^+_1) = \lambda''(U^+,V^+_1) = 0.
\end{equation*}
Then the space $U^+ \oplus V^+_1$ is isotropic for~$\lambda'$ and~$\lambda''$, 
hence $[U^+] \in \Pi$.

This means that the cokernel of~$\lambda_1$ is a line bundle on~$\Pi$.
Therefore, its kernel is a line bundle on~$\bX^+$, and a computation of the determinant shows that it is isomorphic to~$\cO_{\bX^+}(-1)$.
Taking into account the isomorphism $\wedge^2\cU^+ \cong \cO_{\bX^+}(-1)$,
the diagram chase then gives the exact sequence
\begin{equation}
\label{eq:ce-g2}
0 \xrightarrow{\hspace{1em}} \cO_{\bX^+}(-1) \xrightarrow{\hspace{1em}} \cE \xrightarrow{\hspace{1em}} \cO_{\bX^+}(-1) \xrightarrow{\hspace{1em}} \cO_{\Pi}(-1) \xrightarrow{\hspace{1em}} 0.
\end{equation}
It is easy to see that this identifies~$\cE_0$ with the subsheaf~$\cO_{\bX^+}(-1)$.

The fact that the map $\sigma_+ \circ \psi$ is given by the linear system $|\bH - 2\bE|$ follows from Lemma~\ref{lemma:linear-system}.
\end{proof}

\begin{remark}
In this case the natural morphism $\cE_0 \to \cE$ is an embedding of sheaves, but not a fiberwise monomorphism.
As a consequence, the subvariety $\PP_{\bX^+}(\cE_0) \subset \PP_{\bX^+}(\cE)$ is only a rational section
of the $\PP^1$-bundle~$\PP_{\bX^+}(\cE) \to \bX^+$;
in fact, $\PP_{\bX^+}(\cE_0) \cong \Bl_\Pi(\bX^+)$.
\end{remark}

\section{Mukai varieties of genus 7, 8, and 10}

In this section we describe the birational transformations of linear sections of varieties~$\bX$ 
from the previous section and as a consequence prove rationality of higher-dimensional Mukai varieties of genus $g \in \{7,8,10\}$.

\subsection{Forms of linear sections}

We use the notation introduced in~\S\S\ref{subsection:quadric-transformation-statement}--\ref{subsection:quadric-transformation-proof}.
The main result of this section is the following.

\begin{theorem}
\label{theorem:sections}
Let $\bX \subset \PP(W)$ be a smooth projective variety over $\bkk$
and let $\cQ \subset \bX \times \bX^+$ be a flat family of $m$-dimensional quadrics on~$\bX$ 
parameterized by a smooth projective variety~$\bX^+$ 
such that the assumptions of Theorem~\xref{theorem:qk} and Lemma~\xref{lemma:linear-system} are satisfied.
Assume that $X$ is a smooth projective variety over~$\kk$ with a point $x_0 \in X(\kk)$ such that $X_\bkk$ 
is isomorphic to a transverse linear section of $\bX$ of codimension~$c \le m$ and
\begin{equation}
\label{eq:dim-f1}
\dim(\rF_1(X,x_0)) \le \dim(\rF_1(\bX,x_0)) - c + 1.
\end{equation} 
Then there is a diagram 
\begin{equation}
\label{eq:diagram-x}
\vcenter{\xymatrix@C=4em{
\tX \ar[d]_\sigma \ar@{-->}[rr]^\psi \ar[dr]^\phi
&&
\tX^+
\ar[d]^{\sigma_+} \ar[dl]_{\phi_+}
\\
X &
\barX &
\bX^+
}}
\end{equation}
defined over~$\kk$, where 
\begin{itemize}
\item 
$\sigma$ is the blowup of $x_0$, 
\item 
$\sigma_+$ is a projective morphism with general fiber $\PP^{m-c}$;
\item 
$\psi$ is a birational map.
\end{itemize}
In particular, $\bX^+$ is defined over~$\kk$.

Moreover, if~$E \subset \tX$ is the exceptional divisor of~$\sigma$,
its strict transform~$\psi_*(E) \subset \tX^+$ intersects the general fiber of~$\sigma_+$ along a hyperplane defined over~$\kk$.
In particular, $X$ is birational to~\mbox{$\bX^+ \times \PP^{m-c}$}.
\end{theorem}

\begin{proof}
We denote $N := \dim(\bX)$.
First, we construct diagram~\eqref{eq:diagram-x} over $\bkk$ and then check that it is defined over~$\kk$.
So, assume for now that $X$ is a $\bkk$-variety.

By assumption, we have an embedding $X \to \bX$.
Using it we consider~$x_0$ as a point of~$\bX$.
Consider the diagram~\eqref{diagram:typical}
and denote by~$\tX := \Bl_{x_0}(X)$ the strict transform of~$X$ in~$\Bl_{x_0}(\bX)$ 
and by~$\barX := \phi(\tX)$ the image of~$\tX$ in~$\barbX$.

Since~$X$ is a complete intersection in~$\bX$ of divisors in the linear system~$|\bH|$ which pass through the point~$x_0$,
its strict transform~$\tX$ is a complete intersection in~$\Bl_{x_0}(\bX)$ of divisors in the linear system~\mbox{$|\barbH| = |\bH - \bE|$}.
Therefore,
\begin{equation*}
\tX = \phi^{-1}(\barX).
\end{equation*}
Moreover, by~\eqref{eq:iso-bx-open} the morphism $\phi \colon \tX \to \barX$ 
is an isomorphism over the complement of the subscheme $\rF_1(X,x_0) = \rF_1(\bX,x_0) \cap \barX$.
By~\eqref{eq:dim-f1} and assumption~\ref{ass:f1} in Theorem~\ref{theorem:qk} we have
\begin{equation*}
\dim(\rF_1(X,x_0)) \le 
\dim(\rF_1(\bX,x_0)) - c + 1 \le
N - c - 2 =
\dim(\barX) - 2.
\end{equation*}
Furthermore, 
using the fact that~$\phi_+$ is a projective bundle over~$\rF_1(\bX,x_0)$ we conclude that
\begin{multline}
\label{eq:phi-plus-f1}
\dim(\phi_+^{-1}(\rF_1(X,x_0))) \le 
\dim(\phi_+^{-1}(\rF_1(\bX,x_0))) - c + 1 \\ = 
\dim(\rF_1(\cQ/\bX^+,x_0)) - c + 1 =
N - c - 1.
\end{multline}

Let 
\begin{equation*}
\tX^+ := \phi_+^{-1}(\barX) \subset \PP_{\bX^+}(\cE).
\end{equation*}
On the one hand, $\tX^+$ is an intersection of~$c$ divisors from the linear system~$|\barbH|$, 
hence we have~\mbox{$\dim(\tX^+) \ge N - c$}.
On the other hand, by~\eqref{eq:iso-bx-open} 
we have
\begin{equation*}
\tX^+ \setminus \phi_+^{-1}(\rF_1(X,x_0)) = \barX \setminus \rF_1(X,x_0) = X \setminus \cL(X,x_0)
\end{equation*}
is irreducible of dimension $N - c$.
These two observations combined with~\eqref{eq:phi-plus-f1} imply that~$\tX^+$ is irreducible of dimension~$N - c$
(in particular it is Cohen--Macaulay), and that $\psi$ is a birational map.
Similarly, using~\eqref{eq:iso-be-open} instead of~\eqref{eq:iso-bx-open} we deduce that
\begin{equation}
\label{eq:txplus-eplus}
\dim(\tX^+ \cap \PP_{\bX^+}(\cE_0)) = N - c - 1,
\end{equation} 
hence it is a divisor in~$\tX^+$.

Since $\barbH$ is a relative hyperplane section for~$\PP_{\bX^+}(\cE)$, 
the subvariety $\tX^+ \subset \PP_{\bX^+}(\cE)$ corresponds to a morphism
\begin{equation}
\label{eq:xi}
\xi \colon \cE \xlongrightarrow{\hspace{2em}} \cO_{\bX^+}^{\oplus c}
\end{equation}
of vector bundles, and similarly, the divisor $\psi_*(E) = \tX^+ \cap \PP_{\bX^+}(\cE_0)$ corresponds to the morphism
\begin{equation}
\label{eq:xi0}
\xi_0 = \xi\vert_{\cE_0} \colon \cE_0 \xlongrightarrow{\hspace{2em}} \cO_{\bX^+}^{\oplus c}.
\end{equation}
If the morphism $\xi_0$ is not generically surjective then the general fiber of its kernel 
is a vector space of dimension at least $m - c + 1$, hence
\begin{equation*}
\dim(\tX^+ \cap \PP_{\bX^+}(\cE_0)) \ge \dim(\bX^+) + (m - c) = N - c,
\end{equation*}
contradicting to~\eqref{eq:txplus-eplus}.

Thus~$\xi_0$ and hence a fortiori~$\xi$ is generically surjective. 
Therefore, the morphism $\sigma_+ \colon \tX^+ \to \bX^+$ is generically the projectivization of a vector bundle of rank~$m + 1 - c$, 
so the general fiber of~$\sigma_+$ is~$\PP^{m-c}$.
Moreover, the divisor $\PP_{\bX^+}(\cE_0) \subset \PP_{\bX^+}(\cE)$ cuts a hyperplane in the general fiber of~$\sigma_+$.
This shows that we have diagram~\eqref{eq:diagram-x} over~$\bkk$ and proves its properties.

It remains to show that the diagram is defined over~$\kk$ if $X$ is and $x_0 \in X(\kk)$.
First, $X$ and its blowup $\tX$ at $x_0$ are defined over~$\kk$.
Next, the divisor classes $H = \bH\vert_X$ and $E = \bE\vert_\tX$ are defined over~$\kk$,
hence the morphism $\phi$ (given by the linear system $|H-E|$) and its image $\barX$ are defined over~$\kk$.
Similarly, the map $\sigma_+ \circ \psi$ (given by the linear system $|(i-m-1)H - (n-m-2)E|$) and its image $\bX^+$ are defined over~$\kk$.
Finally, the map $\psi$ can be defined as the product $X \dashrightarrow \barX \times \bX^+$ of the two maps above,
hence it is defined over~$\kk$, hence so is its image $\tX^+$.
It also follows that the divisor~\mbox{$\psi_*(E) = \tX^+ \cap \PP_{\bX^+}(\cE_0) \subset \tX^+$} is defined over~$\kk$.

To conclude, we see that $X$ is birational to $\tX^+$, which has a map to $\bX^+$ with general fiber $\PP^{m-c}$ 
and with a relative hyperplane $\psi_*(E)$ defined over~$\kk$, hence is birational to $\bX^+ \times \PP^{m-c}$.
\end{proof}

\begin{remark}
\label{remark:3folds-general}
Suppose the assumptions of Theorem~\ref{theorem:sections} are satisfied, but $c = m + 1$.
Then the same argument proves that the morphism~$\xi$ is generically an isomorphism and~$X$ is birational to~$\tX^+$,
which itself is birational to the discriminant locus~$\fD(\xi) \subset \bX^+$ of the morphism~\eqref{eq:xi}.
Moreover, the morphism $\tX^+ \to \fD(\xi)$ is the Springer (partial) resolution of $\fD(\xi)$.
Using the other Springer resolution, we can construct the subvariety
\begin{equation*}
\tX^{++} \subset \bX^+ \times \PP^m,
\end{equation*}
which is birational to $\tX^+$ (and hence to~$X$) and is equal to the zero locus of the global section
of the vector bundle~$\cE^\vee \boxtimes \cO_{\PP^m}(1)$.
The induced projection $\tX^{++} \to \PP^m$ can provide an additional information about~$X$, 
see Remarks~\ref{remark:x-3-7}--\ref{remark:x-3-10} below.
\end{remark}

\subsection{Rationality of Mukai varieties}

In this section we apply Theorem~\ref{theorem:sections} to Mukai varieties of genus~$g \in \{7,\, 8,\, 10\}$.
We prove the following

\begin{theorem}
\label{theorem:rationality}
Let $X$ be a smooth Mukai variety of genus~$g \in \{7,\, 8,\, 10\}$ and dimension~$n \ge 4$.
If~\mbox{$X(\kk) \ne \varnothing$}, then $X$ is $\kk$-rational.
\end{theorem}

\begin{proof}
Let $X$ be a smooth Mukai variety of genus~$g \in \{7,\, 8,\, 10\}$ and let $\bX = \bX_{2g - 2}$ 
be the maximal Mukai variety over~$\bkk$ of the same genus, i.e., $\bX = \OGr_+(5,10)$, $\Gr(2,6)$, or~$\GTGr(2,7)$.
By Theorem~\ref{theorem-mukai} the variety $X_\bkk$ is a transverse linear section of~$\bX$.
By Propositions~\ref{proposition:gr2v}, \ref{proposition:ogr510}, and~\ref{proposition:g2gr}
the assumptions of Theorem~\ref{theorem:qk} and Lemma~\ref{lemma:linear-system} are satisfied for~$\bX$.
Recall that the corresponding varieties~$\bX^+$ are~$\PP^4$, $\Gr(2,4)$, and the quintic del Pezzo fourfold.

We check below that assumption~\eqref{eq:dim-f1} also holds.
Indeed, assume to the contrary that 
\begin{equation*}
\dim(\rF_1(X,x_0)) \ge \dim(\rF_1(\bX,x_0)) - c + 2.
\end{equation*}
Note that $\dim(\rF_1(\bX,x_0)) = \dim(\bX) - 4$ for each of the maximal Mukai varieties~$\bX$
(this easily follows from Lemmas~\ref{lemma:f1-ogr}, \ref{lemma:f1-gr2m}, and~\ref{lemma:f1-g2}),
hence the above inequality implies that 
\begin{equation*}
\dim(\cL(X,x_0)) = \dim(\rF_1(X,x_0)) + 1 \ge \dim(\bX) - 4 - c + 2 + 1 = \dim(X) - 1.
\end{equation*}
In other words, lines passing through $x_0$ sweep on~$X$ a divisor.
This divisor is automatically contained in the intersection of~$X$ with the embedded tangent space of~$X$ at~$x_0$.
But $\uprho(X) = 1$ by Lefschetz Theorem and $X$ is not a hypersurface, hence this is impossible.

This shows that the assumptions of Theorem~\ref{theorem:sections} are satisfied.
We deduce from it that~$X$ is $\kk$-birational to $\bX^+ \times \PP^{m-c}$.
On the other hand, $\bX^+$ has a $\kk$-point by Nishimura lemma, 
and since it is a $\kk$-form of $\PP^4$, $\Gr(2,4)$, or a quintic del Pezzo fourfold, it follows that $\bX^+$ is $\kk$-rational
(see~\cite[Propositions~2.5 and~2.6 and Remark~3.4]{KP19}).
Consequently, $X$ is $\kk$-rational as well.
\end{proof}

Using Remark~\ref{remark:3folds-general} we can also describe birationally Mukai threefolds of genus~$g \in \{7,\, 8,\, 10\}$.

\begin{remark}
\label{remark:x-3-7}
If $X$ is a Mukai threefold of genus $g = 7$ we conclude from Remark~\ref{remark:3folds-general} that $X$ is birational 
to the zero locus~$\tX^{++}$ of the section~$\xi$ of the vector bundle
\begin{equation*}
\cE^\vee \boxtimes \cO(1) \cong (\Omega^2(3) \boxtimes \cO(1)) \oplus (\cO(1) \boxtimes \cO(1))
\end{equation*}
on $\bX^+ \times \PP^6 = \PP^4 \times \PP^6$.
It is easy to show that the zero locus of the component~$\xi_0$ of $\xi$ in the first summand is isomorphic 
to the projectivization of a rank-2 vector bundle on a (possibly singular) quintic del Pezzo threefold~$X^{++}$
(the number of its singular points is equal to the number of lines on~$X$ through~$x_0$).
It follows that $\tX^{++}$ is isomorphic to the blowup of~$X^{++}$ along a curve; 
in particular, $X$ is birational to~$X^{++}$, hence is $\kk$-rational
(this can be proved by the argument of~\cite[Theorem~3.3]{KP19}).
This provides a more direct proof of the rationality criterion for~$X$ from~\cite{KP19}.
\end{remark}

\begin{remark}
\label{remark:x-3-8}
If $X$ is a Mukai threefold of genus $g = 8$ we conclude that $X$ is birational 
to the zero locus~$\tX^{++}$ of the section $\xi$ of the vector bundle
\begin{equation*}
\cE^\vee \boxtimes \cO(1) \cong (U_0^\vee \otimes (\cU^+)^\vee \boxtimes \cO(1)) \oplus (\cO(1) \boxtimes \cO(1))
\end{equation*}
on $\bX^+ \times \PP^4 = \Gr(2,4) \times \PP^4$.
It is easy to show that the zero locus of the component~$\xi_0$ of~$\xi$ in the first summand is isomorphic 
to the blowup of~$\PP^4$ along a normal rational quartic curve and, furthermore, $\tX^{++}$ 
is isomorphic to a cubic threefold passing through that curve.
\end{remark}

\begin{remark}
\label{remark:x-3-10}
If $X$ is a Mukai threefold of genus $g = 10$ we conclude that $X$ is birational 
to the zero locus~$\tX^{++}$ of the section $\xi$ of the vector bundle~$\cE^\vee \boxtimes \cO(1)$ on~$\bX^+ \times \PP^1$, 
where~$\cE$ is defined by the exact sequence~\eqref{eq:ce-g2}.
It is easy to show that the induced morphism~$\tX^{++} \to \PP^1$ is a fibration in sextic del Pezzo surfaces.
One can check that this morphism is given by the linear system $|H - 3E|$ 
and generalizes the Mori fiber space of~\cite[Theorem~5.18(iii)]{KP19} 
to the case when $\rF_1(X,x_0) \ne \varnothing$.
\end{remark}

\section{Mukai varieties of genus~$9$}
\label{section:genus-9}

In this section we discuss the maximal Mukai variety $\bX = \bX_{16} = \LGr(3,6)$ of genus~$9$, 
its linear sections and their forms over non-closed fields.
First, we construct a Sarkisov link for~$\bX$ starting with the blowup of a point 
(we state the theorem in~\S\ref{subsection:statement-9} and prove it in~\S\ref{subsection:proof-9}), 
analogous to those constructed in Theorem~\ref{theorem:qk}.
After that in~\S\ref{subsection:implications-9} we deduce implications for Mukai varieties of genus~$9$.

\subsection{The statement}
\label{subsection:statement-9}

Let $U_0$ be a $3$-dimensional vector space (later it will correspond to a point of~$\LGr(3,6)$).
Consider the Veronese surfaces
\begin{equation*}
\rS \subset \PP(S^2U_0),
\qquad 
\rSv \subset \PP(S^2U_0^\vee),
\end{equation*}
the images of the double Veronese embeddings $\PP(U_0) \hookrightarrow \PP(S^2U_0)$ and $\PP(U_0^\vee) \hookrightarrow \PP(S^2U_0^\vee)$.
Furthermore, denote by $\bK \subset \PP(S^2U_0)$ and $\bK^\vee \subset \PP(S^2U_0^\vee)$ 
the secant varieties of~$\rS$ and~$\rSv$, respectively
(so-called, {\sf chordal cubic fourfolds}~\cite[3.11-3.12]{Semple-Roth-1949}), so that
\begin{equation}
\label{eq:chordal}
\rS = \Sing(\bK) \subset \bK \subset \PP(S^2U_0),
\qquad 
\rSv = \Sing(\bK^\vee) \subset \bK^\vee \subset \PP(S^2U_0^\vee).
\end{equation}
These geometric data give rise to a classical birational transformation
(see, e.g., \cite[Theorem~3.3]{CrauderKatz89}, \cite[Theorem~6]{AlzatiSierra16}).

\begin{lemma}
\label{lemma:cremona}
There is an isomorphism $\Bl_{\rSv}(\PP(S^2U_0^\vee)) \cong \Bl_{\rS}(\PP(S^2U_0))$
which fits into the following symmetric diagram
\begin{equation*}
\vcenter{\xymatrix{
&
\bFv \ar@{^{(}->}[r] &
\Bl_{\rSv}(\PP(S^2U_0^\vee)) \ar@{=}[r] \ar[dl]_{\kappa^\vee} &
\Bl_{\rS}(\PP(S^2U_0)) \ar[dr]^{\kappa} &
\bF \ar@{_{(}->}[l] 
\\
\rSv \ar@{^{(}->}[r] &
\PP(S^2U_0^\vee) &&&
\PP(S^2U_0) &
\rS, \ar@{_{(}->}[l] 
}}
\end{equation*}
where 
\begin{itemize}
\item 
$\kappa$ and~$\kappa^\vee$ are the blowups of the Veronese surfaces~$\rS$ and~$\rSv$ respectively,
\item 
$\bF$ and~$\bFv$ are the corresponding exceptional divisors,
\item 
the morphisms $\kappa \colon \bF \to \rS$ and $\kappa^\vee \colon \bFv \to \rSv$ are $\PP^2$-bundles,
\item 
$\kappa(\bFv) = \bK$ and $\kappa^\vee(\bF) = \bK^\vee$ and the maps $\kappa \colon \bFv \to \bK$ and $\kappa^\vee \colon \bF \to \bK^\vee$ are the blowups with centers in~$\rS$ and~$\rSv$, respectively.
\end{itemize}

The fibers of the map $\kappa^\vee \colon \bFv \to \rSv$ are mapped by~$\kappa$ to planes in~$\PP(S^2U_0)$
intersecting~$\rS$ along smooth conics, and analogously for the fibers of~$\kappa$.

Finally, the maps 
\begin{equation*}
\kappa^\vee \circ \kappa^{-1} \colon \PP(S^2U_0) \xdashrightarrow{\hspace{2em}} \PP(S^2U_0^\vee)\quad 
\text{and}\quad \kappa \circ (\kappa^\vee)^{-1} \colon \PP(S^2U_0^\vee) \xdashrightarrow{\hspace{2em}} \PP(S^2U_0) 
\end{equation*}
are given by the complete linear systems of quadrics through~$\rS$ and~$\rSv$, respectively.
\end{lemma}

Now let $V$ be a $6$-dimensional symplectic vector space.
Define 
\begin{equation*}
W := \Ker \bigl( \wedge^3V \xrightarrow{\hspace{1em}} V \bigr)
\end{equation*}
(the morphism is given by convolution with the symplectic form).
This is a fundamental $14$-dimensional representation of the symplectic group~$\Sp(V)$,
and 
\begin{equation*}
\bX = \LGr(3,V) \subset \PP(W)
\end{equation*}
is the orbit of the highest weight vector.
We denote by~$\bH$ the restriction to~$\bX$ of the hyperplane class of~$\PP(W)$.

Let $x_0 \in \bX$ be a point and let $U_0 \subset V$ be the corresponding $3$-dimensional isotropic subspace.

\begin{lemma}
\label{lemma:lines-lgr}
There is a natural isomorphism $\rF_1(\bX,x_0) \cong \PP(U_0^\vee)$.
\end{lemma}
\begin{proof}
Recall that any line on~$\Gr(3,V)$ has the form $\ell(V_2,V_4) = \{ [U] \mid V_2 \subset U \subset V_4 \}$,
where \mbox{$V_2 \subset V_4 \subset V$} is a flag with $\dim(V_i) = i$.
Such a line is contained in~$\LGr(3,V)$ if and only if the restriction of the symplectic form to~$V_4$ contains~$V_2$ in the kernel;
equivalently, if~$V_2$ is isotropic and~$V_4$ is the orthogonal of~$V_2$ with respect to the symplectic form.
Furthermore, the line $\ell(V_2,V_2^\perp)$ contains~$[U_0]$ if and only if~\mbox{$V_2 \subset U_0$}.
Therefore, $\rF_1(\bX,x_0) = \Gr(2,U_0) \cong \PP(U_0^\vee)$.
\end{proof}

We choose a Lagrangian direct sum decomposition~$V = U_0 \oplus U_0^\vee$;
it induces a direct sum decomposition of~$W$ that has the following form
\begin{equation}
\label{eq:bw30v6-point}
W \cong \det(U_0) \oplus \bigl(S^2U_0^\vee \otimes \det(U_0)\bigr) \oplus \bigl(S^2U_0 \otimes \det(U_0^\vee)\bigr) \oplus \det(U_0^\vee),
\end{equation}
and the point $x_0$ corresponds to the summand $\det(U_0)$ in the right side.
We consider the blowup~\mbox{$\sigma \colon \Bl_{x_0}(\bX) \to \bX$} and denote by~$\bE$ its exceptional divisor.

Consider the last two summands in~\eqref{eq:bw30v6-point} (we ignore the twist by~$\det(U_0^\vee)$ for simplicity) 
\begin{equation}
\label{eq:wplus}
W^+ := S^2U_0 \oplus \kk.
\end{equation}
We denote by~$\bH^+$ the hyperplane class of~$\PP(W^+)$.
Consider the chain of embeddings 
\begin{equation*}
\rS \xhookrightarrow{\hspace{2em}} \PP(S^2U_0) \xhookrightarrow{\hspace{2em}} \PP(W^+),
\end{equation*}
the blowup $\sigma_+ \colon \Bl_\rS(\PP(W^+)) \to \PP(W^+)$ 
and denote by~$\bE^+$ its exceptional divisor 
and by 
\begin{equation}
\label{eq:bzp}
\bZ^+ \cong \Bl_\rS(\PP(S^2U_0)) \subset \Bl_\rS(\PP(W^+))
\end{equation} 
the strict transform of the hyperplane $\PP(S^2U_0) \subset \PP(W^+)$.
Note that
\begin{equation}
\label{eq:bzp-class}
\bZ^+ \sim \bH^+ - \bE^+.
\end{equation} 

\begin{lemma}
\label{lemma:bpf}
The linear systems $|\bH - \bE|$ on~$\Bl_{x_0}(\bX)$ and $|2\bH^+ - \bE^+|$ on~$\Bl_\rS(\PP(W^+))$ are base point free.
Moreover, if
\begin{equation}
\label{eq:barw-9}
\barW := \bigl(S^2U_0^\vee \otimes \det(U_0)\bigr) \oplus 
\bigl(S^2U_0 \otimes \det(U_0^\vee)\bigr) \oplus \det(U_0^\vee),
\end{equation}
is the quotient of~$W$ by the first summand~$\det(U_0)$, see~\eqref{eq:bw30v6-point}, then
\begin{equation*}
H^0(\Bl_{x_0}(\bX), \cO(\bH - \bE)) \cong \barW{}^\vee \cong H^0(\Bl_\rS(\PP(W^+)), \cO(2\bH^+ - \bE^+)).
\end{equation*}
In particular, the hyperplane $\barW{}^\vee \subset W^\vee$ corresponds to the point~$x_0 \in \LGr(3,U_0 \oplus U_0^\vee) \subset \PP(W)$.
\end{lemma}

\begin{proof}
The first statement is easy, because the point $x_0 \in \bX$ and the surface $\rS \subset \PP(W^+)$
are intersections of hyperplanes (resp.\ quadrics) in~$\bX$ (resp.\ in~$\PP(W^+)$), as schemes.

Furthermore, the first of the isomorphisms follows immediately from~\eqref{eq:bw30v6-point} and~\eqref{eq:barw-9}.
For the second consider the exact sequence (here we use~\eqref{eq:bzp-class})
\begin{equation}
\label{eq:sequence-2h-e}
0 
\xrightarrow{\hspace{1em}}
\cO_{\Bl_\rS(\PP(W^+))}(\bH^+) \xrightarrow{\ \bZ^+\ } \cO_{\Bl_\rS(\PP(W^+))}(2\bH^+ - \bE^+) 
\xrightarrow{\hspace{1em}}
\cO_{\Bl_\rS(\PP(S^2U_0))}(2\bH^+ - \bE^+) 
\xrightarrow{\hspace{1em}}
0.
\end{equation}
The cohomology of the first term is $(W^+)^\vee$, and the cohomology of the last is~$S^2U_0$ (by~\eqref{eq:bzp} and Lemma~\ref{lemma:cremona}).
Summing up and dualizing, we obtain the required isomorphism.
\end{proof}

Now we are ready to state the theorem.
Recall that $\rF_1(\bX,x_0)$ denotes the Hilbert scheme of lines on~$\bX$ passing through~$x_0$ 
and~$\cL(\bX,x_0)$ is the corresponding universal line.

\begin{theorem}
\label{theorem:lg36}
Let $\bX = \LGr(3,V)$.
There is a commutative diagram 
\begin{equation}
\label{diagram:lgr36}
\vcenter{\xymatrix@C=4em{
&
\hbD \ar@{^{(}->}[d] \ar[dl]_p \ar[dr]^{p_+}
\\
\cL(\bX,x_0) \ar@{^{(}->}[d] &
\hbX \ar[dl]^{\pi} \ar[dr]_{\pi_+} &
\bFv \ar@{^{(}->}[d]
\\
\Bl_{x_0}(\bX) \ar[d]_\sigma \ar@{-->}[rr]^\psi \ar[dr]^\phi
&&
\Bl_{\rS}(\PP(W^+)) \ar[d]^{\sigma_+} \ar[dl]_{\phi_+}
\\
\bX &
\barbX &
\PP(W^+),
}}
\end{equation}
where 
\begin{itemize}
\item 
$\sigma$ is the blowup of the point $x_0$,

\item 
$\sigma_+$ it the blowup of the Veronese surface~$\rS \subset \PP(W^+)$, where $W^+$ is defined in~\eqref{eq:wplus},
\item 
$\phi$ is induced by the linear projection from~$x_0$, 
\item 
$\phi_+$ is the morphism given by the linear system of quadrics in~$\PP(W^+)$ through~$\rS$, 
\item 
$\barbX = \phi(\Bl_{x_0}(\bX)) = \phi_+(\Bl_{\rS}(\PP(W^+))) \subset \PP(\barW)$,
where $\barW$ is defined in~\eqref{eq:barw-9},
\item 
$\pi$ is the blowup of $\cL(\bX,x_0) \subset \Bl_{x_0}(\bX)$, 
\item 
$\pi_+$ is the blowup of $\bFv \subset \Bl_{\rS}(\PP(W^+))$, 
\item 
$\hbD$ is the common exceptional divisor of $\pi$ and $\pi_+$,
\begin{equation*}
\hbD \cong \cL(\bX,x_0) \times_{\rSv} \bFv,
\end{equation*}
\item 
$\psi = \pi_+ \circ \pi^{-1} = \phi_+^{-1} \circ \phi$ is an antiflip.
\end{itemize}
Moreover, if $\bH$ and $\bH^+$ are the pullbacks to~$\Bl_{x_0}(\bX)$ and~$\Bl_{\rS}(\PP(W^+))$
of the hyperplane classes of~$\bX$ and~$\PP(W^+)$, respectively,
and~$\bE \subset \Bl_{x_0}(\bX)$ and~$\bE^+ \subset \Bl_{\rS}(\PP(W^+))$ are the exceptional divisors of~$\sigma$ and~$\sigma_+$,
then there are the following relations 
\begin{equation}
\label{eq:relations}
\left\{\begin{array}{lcl}
\bH^+ &=& \bH - 2\bE,\\
\bE^+ &=& \bH - 3\bE,
\end{array}\right.
\qquad\text{and}\qquad
\left\{\begin{array}{lcl}
\bH &=& 3\bH^+ - 2\bE^+,\\
\bE &=& \hphantom{1}\bH^+ - \hphantom{1}\bE^+, 
\end{array}\right.
\end{equation}
in $\Pic(\Bl_{x_0}(\bX)) = \Pic(\Bl_{\rS}(\PP(W^+)))$, identified via~$\psi$.
In particular, the hyperplane class $\barbH$ of~$\PP(\barW)$ when pulled back to~$\Bl_{x_0}(\bX)$ and~$\Bl_{\rS}(\PP(W^+))$ can be written as
\begin{equation*}
\barbH = \bH - \bE = 2\bH^+ - \bE^+.
\end{equation*}
\end{theorem}

\subsection{The proof}
\label{subsection:proof-9}

In this section we prove Theorem~\ref{theorem:lg36}.
It will be more convenient to construct the diagram ``from right to left''.
We start with some preparations.

\begin{lemma}
\label{lemma:tilde-alpha}
The linear system $|3\bH^+ - 2\bE^+|$ defines a birational map 
\begin{equation}
\label{eq:tilde-alpha}
\tilde\alpha \colon \Bl_{\rS}(\PP(W^+)) \xdashrightarrow{\hspace{2em}} \LGr(3,U_0 \oplus U_0^\vee).
\end{equation} 
\end{lemma}

\begin{proof}
Consider the matrix
\begin{equation}
\label{eq:alpha-matrix}
\alpha = 
\begin{pmatrix}
u_0 & 0 & 0 & u_{11} & u_{12} & u_{13} \\
0 & u_0 & 0 & u_{12} & u_{22} & u_{23} \\
0 & 0 & u_0 & u_{13} & u_{23} & u_{33} 
\end{pmatrix},
\end{equation}
where $(u_0:u_{11}:u_{12}:u_{13}:u_{22}:u_{23}:u_{33})$ are homogeneous coordinates on $\PP(W^+)$ such that~$u_0 = 0$ 
is the equation of the hyperplane $\PP(S^2U_0) \subset \PP(W^+)$ 
and the restriction of $u_{ij}$ to this hyperplane form a standard system of coordinates, 
in which $\rS$ is defined by minors of size~$2$ of the submatix of~$\alpha$ formed by the last three columns
(and consequently, the chordal cubic $\bK \subset \PP(S^2U_0)$ is defined by the determinant of that submatix).
The matrix~$\alpha$ defines a rational map 
\begin{equation*}
\PP(W^+) \xdashrightarrow{\hspace{2em}} \Gr(3,6) = \Gr(3,U_0 \oplus U_0^\vee)
\end{equation*}
which we also denote by~$\alpha$.

The restriction of~$\alpha$ to the affine space $S^2U_0 = \PP(W^+) \setminus \PP(S^2U_0) = \{u_0 \ne 0\} \cong \mathbb{A}^6$
takes a point of~$S^2U_0$ to the graph of the corresponding symmetric morphism $U_0^\vee \to U_0$.
This graph is a Lagrangian subspace for the natural symplectic form on $U_0 \oplus U_0^\vee$, 
hence $\alpha$ induces an isomorphism of~$\PP(W^+) \setminus \PP(S^2U_0)$ onto the open Schubert cell in~$\LGr(3,U_0 \oplus U_0^\vee)$,
parameterizing Lagrangian subspaces which do not intersect~$U_0$.
This also proves that~$\alpha$ factors as the composition
\begin{equation*}
\PP(W^+)\xdashrightarrow{\hspace{2em}}\LGr(3,U_0 \oplus U_0^\vee) \subset \Gr(3,U_0 \oplus U_0^\vee)
\end{equation*}
where the first arrow is birational.

Furthermore, note that Pl\"ucker coordinates on~$\Gr(3,6)$ correspond to the minors of size~3 of~\eqref{eq:alpha-matrix};
these are cubic polynomials in coordinates that can be written as
\begin{equation}
\label{eq:sections}
u_0^3,\quad u_0^2u_{ij},\quad u_0(u_{ij}u_{i'j'} - u_{ij'}u_{i'j}),\quad \det(u_{ij}),
\qquad 
1 \le i \le j \le 3,\ 1 \le i' \le j' \le 3,
\end{equation}
(note that the four groups above correspond to the four summands in~\eqref{eq:bw30v6-point} taken in the opposite order).
It is straightforward to see that each of these polynomials vanishes with multiplicity at least~$2$ on~$\rS$,
hence they generate a subsystem of the linear system $|3\bH^+ - 2\bE^+|$ on $\Bl_\rS(\PP(W^+))$.
Moreover, the cohomology exact sequences associated with the exact sequences of sheaves~\eqref{eq:sequence-2h-e} 
and with a similar sequence
\begin{equation}
\label{eq:sections-3h-2e}
0 \xrightarrow{\hspace{1em}} \cO_{\Bl_\rS(\PP(W^+))}(2\bH^+ - \bE^+)
\xrightarrow{\ \bZ^+\ } \cO_{\Bl_\rS(\PP(W^+))}(3\bH^+ - 2\bE^+)
\xrightarrow{\hspace{1em}} \cO_{\bZ^+}(3\bH^+ - 2\bE^+)
\xrightarrow{\hspace{1em}} 0
\end{equation}
show that the polynomials~\eqref{eq:sections} generate the complete linear system $|3\bH^+ - 2\bE^+|$.
It follows that the target of the rational map defined by this linear system coincides with~$\LGr(3,6)$.
\end{proof}

Recall the subvariety~$\bZ^+ \subset \Bl_\rS(\PP(W^+))$ defined in~\eqref{eq:bzp}.
From the description of Lemma~\ref{lemma:cremona} we know that it contains 
a smooth fourfold $\bFv \subset \bZ^+$.
Furthermore, 
in the Picard group of~$\bZ^+$ we have the following linear equivalence
\begin{equation}
\label{eq:bfv-class}
\bFv \sim (3\bH^+ - 2\bE^+)\vert_{\bZ^+}.
\end{equation} 

Now consider the blowup 
\begin{equation*}
\pi_+ \colon \hbX := \Bl_{\bFv}(\Bl_\rS(\PP(W^+))) \xlongrightarrow{\hspace{2em}} \Bl_\rS(\PP(W^+))
\end{equation*}
and let $\hbD \subset \hbX$ be the exceptional divisor.
Since $\bFv$ is a divisor in~$\bZ^+$, the strict transform of~$\bZ^+$ is isomorphic to~$\bZ^+$; we denote it by
\begin{equation*}
\hbZ \subset \hbX,
\qquad 
\hbZ \cong \bZ^+.
\end{equation*} 
Note that
\begin{equation}
\label{eq:hbz}
\hbZ \sim \bH^+ - \bE^+ - \hbD.
\end{equation} 

\begin{lemma}
The base locus of the linear system $|3\bH^+ - 2\bE^+|$ on~$\Bl_\rS(\PP(W^+))$ is equal to~$\bFv$ as a scheme.
In particular, the rational map~$\tilde\alpha$ defined in~\eqref{eq:tilde-alpha} lifts to a regular birational morphism 
\begin{equation*}
\hat\alpha \colon \hbX \xrightarrow{\hspace{2em}} \LGr(3,U_0 \oplus U_0^\vee).
\end{equation*}
such that $\hat\alpha^*\cO(\bH) \cong \cO(3\bH^+ - 2\bE^+ - \hbD)$.
The scheme preimage of the point $x_0 \in \LGr(3,U_0 \oplus U_0^\vee)$ under~$\hat\alpha$ is the divisor~$\hbZ \subset \hbX$.
\end{lemma}

\begin{proof}
Since the linear system $|2\bH^+ - \bE^+|$ is base point free (Lemma~\ref{lemma:bpf}), 
it follows from sequence~\eqref{eq:sections-3h-2e} 
that the base locus of the linear system~$|3\bH^+ - 2\bE^+|$ on~$\Bl_\rS(\PP(W^+))$ 
is equal to the base locus of the linear system~$|3\bH^+ - 2\bE^+|$ on~$\bZ^+$.
From~\eqref{eq:bfv-class} and Lemma~\ref{lemma:cremona} we deduce that this is~$\bFv$.
Therefore, the linear system~$|3\bH^+ - 2\bE^+ - \hbD|$ on the blowup~$\hbX$ is base point free and defines the morphism~$\hat\alpha$,
such that~$\hat\alpha^*\cO(\bH) \cong \cO(3\bH^+ - 2\bE^+ - \hbD)$.
The morphism~$\hat\alpha$ is birational because~$\tilde\alpha$ is (Lemma~\ref{lemma:tilde-alpha}).

For the last statement recall that by Lemma~\ref{lemma:bpf} the point $x_0 \in \bX$ 
corresponds to the hyperplane~$\barW{}^\vee \subset W^\vee = H^0(\bX,\cO(\bH))$.
Therefore, the scheme $\hat\alpha^{-1}(x_0) \subset \hbX$
is equal to the base locus of the corresponding subsystem in~$|3\bH^+ - 2\bE^+ - \hbD|$.

On the other hand, the proof of the equality $W^\vee = H^0(\Bl_S(\PP(W^+)), \cO(3\bH^+ - 2\bE^+))$ in Lemma~\ref{lemma:tilde-alpha} shows that the subspace $\barW{}^\vee \subset W^\vee$ 
is the image of the top arrow in
the commutative diagram
\begin{equation*}
\xymatrix@C=3em{
H^0(\Bl_\rS(\PP(W^+)),\cO(2\bH^+ - \bE^+)) \ar@{=}[d] \ar[r]^{\bZ^+} &
H^0(\Bl_\rS(\PP(W^+)),\cO(3\bH^+ - 2\bE^+)) \ar@{=}[d]
\\
H^0(\hbX,\cO(2\bH^+ - \bE^+)) \ar[r]^{\hbZ} &
H^0(\hbX,\cO(3\bH^+ - 2\bE^+ - \hbD)), 
}
\end{equation*}
hence coincides with the image of the bottom map.
Since the linear system $|2\bH^+ - \bE^+|$ is base point free (Lemma~\ref{lemma:bpf}), it follows that
the base locus of the corresponding subsystem in~$|3\bH^+ - 2\bE^+ - \hbD|$ is~$\hbZ$, hence $\hat\alpha^{-1}(x_0) = \hbZ$.
\end{proof}

Since the preimage of the point~$x_0$ is the Cartier divisor $\hbZ \subset \hbX$, we obtain from~\eqref{eq:hbz} the following

\begin{corollary}
\label{corollary:pi-9}
The morphism $\hat\alpha$ factors as the composition
\begin{equation*}
\hbX \xrightarrow{\ \pi\ } \Bl_{x_0}(\bX) \xrightarrow{\ \sigma\ } \bX,
\end{equation*}
where $\sigma$ is the blowup of the point~$x_0$ and~$\pi$ is birational.
Moreover,
\begin{equation}
\label{eq:pi-pullback}
\pi^*\cO(\bH) \cong \cO(3\bH^+ - 2\bE^+ - \hbD),
\qquad 
\pi^*\cO(\bE) \cong \cO(\bH^+ - \bE^+ - \hbD).
\end{equation} 
\end{corollary}

Finally, we describe the divisor~$\hbD \subset \hbX$.
We denote by~$h'$ the line class of~$\rSv$.

\begin{lemma}
\label{lemma:hbd-9}
The divisor $\hbD$ is isomorphic to the fiber product 
\begin{equation}
\label{eq:hbd}
\hbD \cong 
\PP_{\rSv}(\cO \oplus \cO(-2h')) \times_\rSv \bFv.
\end{equation}
In particular, it has a structure of a $(\PP^1 \times \PP^2)$-fibration over~$\rSv$.
Moreover, we have
\begin{equation}
\label{bidegree}
\deg(\bH^+\vert_{\hbD}) = (0,1),
\qquad 
\deg(\bE^+\vert_{\hbD}) = (0,2),
\qquad 
\deg(\hbD\vert_{\hbD}) = (-1,-1),
\qquad 
\end{equation} 
where $\deg(-)$ stands for the bidegree on the fibers of the $(\PP^1 \times \PP^2)$-fibration.
\end{lemma}

\begin{proof}
Recall that $\hbD$ is the exceptional divisor of the blowup~$\pi_+$, hence
\begin{equation*}
\hbD \cong \PP_{\bFv}(\cN_{\bFv/\Bl_\rS(\PP(W^+))}).
\end{equation*}
The chain of embeddings $\bFv \subset \bZ^+ \subset \Bl_\rS(\PP(W^+))$ in view of~\eqref{eq:bfv-class} and~\eqref{eq:bzp-class}
gives rise to the normal bundle sequence
\begin{equation*}
0 \xrightarrow{\hspace{1em}} \cO_\bFv(3\bH^+ - 2\bE^+) \xrightarrow{\hspace{1em}} \cN_{\bFv/\Bl_\rS(\PP(W^+))} \xrightarrow{\hspace{1em}} \cO_\bFv(\bH^+ - \bE^+) \xrightarrow{\hspace{1em}} 0.
\end{equation*}
By Lemma~\ref{lemma:cremona} the divisor class $2\bH^+ - \bE^+$ is the pullback of $2h'$ from~$\rSv$.
Therefore, after a twist the sequence can be rewritten as 
\begin{equation*}
0 \xrightarrow{\hspace{1em}} \cO_\bFv \xrightarrow{\hspace{1em}} \cN_{\bFv/\Bl_\rS(\PP(W^+))}(2\bE^+ - 3\bH^+) \xrightarrow{\hspace{1em}} \cO_\bFv(-2h') \xrightarrow{\hspace{1em}} 0.
\end{equation*}
We have $\Ext^1(\cO_\bFv(-2h'),\cO_\bFv) = H^1(\bFv,\cO_\bFv(2h')) = H^1(\rSv,\cO_\rSv(2h')) = 0$, hence 
\begin{equation}
\label{eq:normal-bfv}
\cN_{\bFv/\Bl_\rS(\PP(W^+))}(2\bE^+ - 3\bH^+) \cong \cO_\bFv \oplus \cO_\bFv(-2h'),
\end{equation}
which proves~\eqref{eq:hbd}.
Thus, $\hbD$ is a $(\PP^1\times \PP^2)$-fiber bundle over~$\rSv$.

By Lemma~\ref{lemma:cremona} the fibers of~$\bFv$ over~$\rSv$ project to (linearly embedded) planes in $\PP(W^+)$
which intersect~$\rS$ along conics, therefore the classes~$\bH^+$ and~$\bE^+$ 
have bidegrees~$(0,1)$ and~$(0,2)$ on the fibers of~$\hbD$.
To find the bidegree of~$\hbD\vert_{\hbD}$ we compute the canonical class of~$\hbD$.

By the blowup formula (applied to~$\sigma_+$ and~$\pi_+$) and adjunction we have
\begin{equation*}
K_{\hbD} = (-7\bH^+ + 3\bE^+ + 2\hbD)\vert_{\hbD}.
\end{equation*}
Since the canonical class has bidegree~$(-2,-3)$, the result follows.
\end{proof}

Now we are ready to prove the theorem.

\begin{proof}[Proof of Theorem~\textup{\ref{theorem:lg36}}]
Consider the morphism $\pi \colon \hbX \to \Bl_{x_0}(\bX)$ constructed in Corollary~\ref{corollary:pi-9}.
This is a birational morphism of smooth varieties, hence its exceptional locus is a divisor.
Since
\begin{equation*}
K_{\hbX} = -7\bH^+ + 3\bE^+ + \hbD,
\qquad 
K_{\Bl_{x_0}(\bX)} = -4\bH + 5\bE,
\end{equation*}
it follows from~\eqref{eq:pi-pullback} that $K_{\hbX/\Bl_{x_0}(\bX)} = 2\hbD$.
But~$\hbD$ is the exceptional divisor of~$\pi_+$, hence any its multiple is non-movable.
Therefore, the exceptional divisor of~$\pi$ is also~$\hbD$.

Since~$\pi$ and~$\pi_+$ have the same exceptional divisor, the birational map
\begin{equation*}
\psi := \pi_+ \circ \pi^{-1} \colon \Bl_{x_0}(\bX) \xdashrightarrow{\hspace{2em}} \Bl_\rS(\PP(W^+))
\end{equation*}
is small. 
In particular, it identifies the Picard groups of $\Bl_{x_0}(\bX)$ and $\Bl_\rS(\PP(W^+))$
(with the quotient of the Picard group of~$\hbX$ by the class of~$\hbD$).
The relations~\eqref{eq:relations} follow easily from~\eqref{eq:pi-pullback}.
Moreover, the above calculation of the canonical classes shows that~$\psi$ is an antiflip.

The relations~\eqref{eq:relations} also imply the equality $\bH - \bE = 2\bH^+ - \bE^+$ in~$\Pic(\hbX)$.
Therefore, for the morphisms~$\phi \colon \Bl_{x_0}(\bX) \to \PP(W^+)$ and $\phi_+ \colon \Bl_\rS(\PP(W^+)) \to \PP(W^+)$
defined by the base point free linear systems $|\bH - \bE|$ and $|2\bH^+ - \bE^+|$, respectively (see Lemma~\ref{lemma:bpf}), we have
\begin{equation*}
\phi \circ \pi = \phi_+ \circ \pi_+.
\end{equation*}
Hence the middle square in diagram~\eqref{diagram:lgr36} commutes.

It remains to show that~$\pi$ is the blowup of $\cL(\bX,x_0)$ (the embedding $\cL(\bX,x_0) \hookrightarrow \Bl_{x_0}(\bX)$
is constructed in the same way as in Lemma~\ref{lemma:varpi}).

Denote by $p \colon \hbD \to \PP_{\rSv}(\cO \oplus \cO(-2h'))$ and $p_+ \colon \hbD \to \bFv$ the projections
of the fiber product~\eqref{eq:hbd} to the factors.
Note that $p$ is a $\PP^2$-bundle and $p_+$ is a $\PP^1$-bundle.
It follows from~\eqref{eq:pi-pullback} and~\eqref{bidegree}
that 
\begin{equation*}
\deg(\bH\vert_{\hbD}) = \deg(\bE\vert_{\hbD}) = (1,0).
\end{equation*}
Since $\bH$ and $\bE$ generate~$\Pic(\Bl_{x_0}(\bX))$, we conclude that
the restriction of~$\pi$ to~$\hbD$ factors through the contraction~$p$
and takes the fibers of~$p_+$ to strict transforms of lines on~$\bX$ passing through~$x_0$.
Therefore, it induces a morphism $\rSv \to \rF_1(\bX,x_0)$ which induces a surjective map
\begin{equation*}
\cL(\bX,x_0) \times_{\rF_1(\bX,x_0)} \rSv \xrightarrow{\hspace{2em}} \pi(\hbD).
\end{equation*}
Since $\rF_1(\bX,x_0) \cong \PP^2 \cong \rSv$ and the morphism $\rSv \to \rF_1(\bX,x_0)$ 
has connected fibers by Zariski's main theorem, it follows that either this map is constant or is an isomorphism.
In the first case~$\pi(\hbD)$ is a single line that contradicts the formula $K_{\hbX/\Bl_{x_0}(\bX)} = 2\hbD$ proved above.
Therefore
\begin{equation*}
\pi(\hbD) = \cL(\bX,x_0) \subset \Bl_{x_0}(\bX).
\end{equation*}
Finally, applying~\cite[Lemma~2.5]{K18} we conclude that $\pi$ is the blowup with center 
$\cL(\bX,x_0)$.
\end{proof}

\subsection{Implications for genus~$9$ Mukai varieties}
\label{subsection:implications-9}

In this subsection we describe the birational transformations of linear sections of~$\bX = \LGr(3,6)$ 
induced by the link~\eqref{diagram:lgr36},
and as a consequence prove rationality of higher-dimensional Mukai varieties of genus $g = 9$.

\begin{theorem}
\label{theorem:mukai-9}
If $X$ is a Mukai variety of genus~$9$ and dimension~$n \ge 4$ over~$\kk$ such that~$X(\kk) \ne \varnothing$
then $X$ is $\kk$-birational to a normal complete intersection $X^+ \subset \PP^6$ of $6-n$ quadrics
which contain a Veronese surface~$\rS \subset X^+$ defined over~$\kk$ 
such that the divisor~$X^+ \cap \langle \rS \rangle$ is $\kk$-rational.
\end{theorem}

\begin{proof}
The argument of Theorem~\ref{theorem:sections} (using link~\eqref{diagram:lgr36} instead of~\eqref{diagram:typical}) shows that there is a diagram
\begin{equation}
\vcenter{\xymatrix@C=4em{
\tX \ar[d]_\sigma \ar@{-->}[rr]^\psi \ar[dr]^\phi
&&
\tX^+
\ar[d]^{\sigma_+} \ar[dl]_{\phi_+}
\\
X &
\barX &
X^+
}}
\end{equation}
defined over~$\kk$, where $\tX = \Bl_{x_0}(X)$, $\barX = \phi(\tX)$ is a linear section of codimension~$6 - n$ of~\mbox{$\barbX \subset \PP(\barW)$}, and $\tX^+ = \phi_+^{-1}(\barX)$;
the required inequality~\eqref{eq:dim-f1}, which in this case takes the form
\begin{equation*}
\dim(\rF_1(X,x_0)) \le n - 3,
\end{equation*}
is proved by the argument of Theorem~\ref{theorem:rationality}.

By~\eqref{eq:relations} the subvariety~$\tX^+ \subset \Bl_\rS(\PP(W^+))$ 
is a complete intersection of~$6 - n$ divisors from the linear system $|2\bH^+ - \bE^+|$.
Therefore $X^+ = \sigma_+(\tX^+)$ is a complete intersection of $6-n$ quadrics which contain the Veronese surface~$\rS$.
Moreover, $\rS \subset X^+$ is the fundamental locus of~$\sigma_+$, hence is defined over~$\kk$.

Similarly, the divisor $X^+ \cap \langle \rS \rangle$ is the image of the divisor $\tX^+ \cap \bZ^+$,
hence by~\eqref{eq:relations} it is the strict transform of the exceptional divisor~$E$ of~$\sigma$. 
Therefore, it is $\kk$-rational.

It remains to prove that $X^+$ is normal.
If $n = 6$, $X^+ = \PP^6$ and if $n = 5$, $X^+$ is an irreducible quadric, so there is nothing to prove.

Finally, assume that $n = 4$ and $X^+$ is not normal.
Since $X^+ \subset \PP^6$ is a complete intersection, we have~\mbox{$\dim(\Sing(X^+)) = 3$}.
For a general $\PP^3 \subset \PP^6$ the intersection $X^+ \cap \PP^3$ is an irreducible curve of arithmetic genus~1, hence
\begin{equation*}
\Sing(X^+ \cap \PP^3) \subset \Sing(X^+) \cap \PP^3
\end{equation*}
is a single point. 
Therefore, $\Sing(X^+)$ is a linear $3$-space in~$\PP^6$.

On the other hand, Theorem~\ref{theorem:lg36} implies that $X^+ \setminus \bK$ is isomorphic to an open subscheme of~$X$, hence smooth,
hence $\Sing(X^+) \subset X^+ \cap \bK$.
But the chordal cubic $\bK$ does not contain linear spaces of dimension~3.
This contradiction shows that $X^+$ is normal.
\end{proof}

\begin{corollary}
\label{corollary:rationality-9}
Let $X$ be a Mukai variety of genus~$9$ over~$\kk$ such that~$X(\kk) \ne \varnothing$.
If $\dim(X) \ge 5$ then $X$ is $\kk$-rational, and if $\dim(X) = 4$ then $X$ is $\kk$-unirational.
\end{corollary}
\begin{proof}
Theorem~\ref{theorem:mukai-9} proves that $X$ is $\kk$-birational to the intersection of quadrics~$X^+$, so it remains to show that $X^+$ 
is $\kk$-rational or $\kk$-unirational.

If $\dim(X) = 6$ then $X^+ = \PP^6$, hence is $\kk$-rational.
Furthermore, if $\dim(X) \in \{4,5\}$ then $X^+$ is normal, hence $\codim_{X^+}(\Sing(X^+)) \ge 2$.
Moreover, $X^+$ contains a $\kk$-rational hyperplane section, hence has a smooth $\kk$-point.
If $\dim(X) = 5$, so that $X^+$ is a quadric, it follows that $X^+$ is $\kk$-rational,
and if~$\dim(X) = 4$, so that $X^+$ is an intersection of two quadrics, it is $\kk$-unirational by~\cite[Remark~3.28.3]{CTSSD1}.
\end{proof}

It is unclear if a Mukai fourfold~$X$ of genus~$9$ with $X(\kk) \ne \varnothing$ is always~$\kk$-rational;
in fact we expect that this is not the case.
However, we can prove the following sufficient condition

\begin{corollary}
\label{corollary:rationality-9-4}
Let $X$ be a Mukai fourfold of genus~$9$ over~$\kk$ such that~$X(\kk) \ne \varnothing$.
If~$X$ has a curve of odd degree defined over~$\kk$ then $X$ is $\kk$-rational.
\end{corollary}

\begin{proof}
Let $C \subset X$ be a curve of odd degree.
Since~$X$ is $\kk$-unirational by Corollary~\ref{corollary:rationality-9}, we can choose a point~$x_0 \in X(\kk)$ 
such that~$X^+$ is smooth and  the fundamental locus of the map $\sigma_+ \circ \psi \circ \sigma^{-1}$ 
does not intersect~$C$.
Then the image of~$C$ in~$X^+$ is a curve of odd degree.
Therefore, $X^+$ is~$\kk$-rational by~\cite[Theorem~14]{HT}.
\end{proof}

\begin{remark}
Let~$Y \subset \PP^6$ be a smooth complete intersection of two quadrics.
It is easy to check that the Hilbert scheme of Veronese surfaces contained in~$Y$ is $3$-dimensional.
It would be interesting to describe this Hilbert scheme explicitly.
\end{remark}

\section{Fano threefolds of genus~12}
\label{section:genus-12}

In this section we give a new proof of the following theorem from~\cite{KP19}.

\begin{theorem}
\label{theorem:v22}
Let $X$ be a prime Fano threefold of genus~$12$ over a field $\kk$ of characteristic~$0$.
If~$X(\kk) \ne \varnothing$ then $X$ is $\kk$-rational.
\end{theorem}

Recall that over an algebraically closed field any (smooth) prime Fano threefold $X$ of genus~12
is the zero locus of a section of the vector bundle $\wedge^2\cU^\vee \oplus \wedge^2\cU^\vee \oplus \wedge^2\cU^\vee$
on $\bX = \Gr(3,V)$, where~$V$ is a vector space of dimension~7 and $\cU$ is the tautological rank-3 vector bundle.
In~\S\ref{subsection:bundles-12} we show that the same is true over any field of characteristic zero, 
then in~\S\ref{subsection:link-gr37-point} we construct a birational transformation for~$\bX$, 
and in \S\ref{subsection:v22-transformation} we check that it induces a birational map between~$X$ and~$\PP^3$.

\subsection{Vector bundles and Grassmannian embedding}
\label{subsection:bundles-12}

We will need a version of~\cite[Lemma~2.4]{KP19} for vector bundles.
Recall that~$\chi(\cE)$ denotes the Euler characteristic of a coherent sheaf 
and that a vector bundle is called {\sf simple} if it has no non-scalar endomorphisms. 

\begin{proposition}
\label{proposition:vb}
Let $X$ be a scheme over $\kk$ 
and let $\cE$ be a simple vector bundle on the scheme~$X_\bkk$ such that $g^*\cE \cong \cE$ for any $g \in \Gal$.
Then there is a Brauer class $\beta \in \Br(\kk)$ and a $\beta$-twisted vector bundle $\cE_0$ on $X$ such that $\cE \cong (\cE_0)_\bkk$.

Moreover, if $X$ is projective then
\begin{align}
\label{eq:beta-chi}
\chi(\cE) \cdot \beta &= 0,\\
\intertext{and if the line bundle $\det(\cE) \in \Pic(X_\bkk)^\Gal$ is defined over~$\kk$ then also}
\label{eq:beta-rank}
\rank(\cE)\cdot \beta &= 0.
\end{align} 
\end{proposition}
\begin{proof}
For each $g \in \Gal$ we choose an isomorphism $\phi_g \colon g^*\cE \xrightarrow{\ \sim\ } \cE$.
Then for each pair~$(g_1,g_2)$ of elements in $\Gal$ we have two isomorphisms
\begin{equation*}
\phi_{g_1g_2},\ \phi_{g_2}\circ g_2^*\phi_{g_1} \colon
(g_1g_2)^*\cE \cong g_2^*g_1^*\cE \xlongrightarrow{\hspace{2em}} \cE.
\end{equation*}
Since $\cE$ is simple, these two isomorphisms differ by a nonzero scalar, hence there is a collection 
\begin{equation*}
\{ c_{g_1,g_2} \in \bkk^\times \}_{g_1,g_2 \in \Gal}
\quad\text{such that}\quad 
\phi_{g_2}\circ g_2^*\phi_{g_1} = c_{g_1,g_2} \phi_{g_1g_2}.
\end{equation*}
It is easy to see that $\{c_{g_1,g_2}\}$ is a $2$-cocycle for the Galois action on $\bkk^\times$,
and a different choice of~$\phi_g$ results in replacing this cocycle by a coboundary.
If $\beta \in H^2(\Gal,\bkk^\times) = \Br(\kk)$ is the cohomology class of this cocycle, 
then by Galois descent the collection of isomorphisms $\phi_g$ 
defines a $\beta$-twisted vector bundle $\cE_0$ on $X$ with the required property.

The proof of property~\eqref{eq:beta-chi} is analogous to the proof of~\cite[Lemma~2.4]{KP19};
it is based on the fact, that after sufficiently ample twist the global sections of the bundle~$\cE$
form a $\beta$-twisted vector space of dimension $\chi(\cE)$, hence $\beta$ is killed by $\chi(\cE)$.

Finally, the isomorphism $\cE \cong (\cE_0)_\bkk$ implies $\det(\cE) \cong \det(\cE_0)_\bkk$.
Note that~$\det(\cE_0)$ is naturally~$(\rank(\cE) \cdot\beta)$-twisted.
Therefore, if $\det(\cE)$ is defined over~$\kk$, then $\rank(\cE) \cdot \beta = 0$.
\end{proof}

Recall from~\cite{KPS} the following result.

\begin{lemma}[{\cite[Theorem~B.1.1, Proposition~B.1.5, and Lemma~B.1.9]{KPS}}]
\label{lemma:vb-e}
Let $X$ be a prime Fano threefold of genus~$12$ over an algebraically closed field of characteristic~$0$.
There is a unique stable globally generated vector bundle~$E$ of rank~$2$ on~$X$
with $\rc_1(E) = 1$, $\rc_2(E) = 7$, such that $\cO$ and $E$ form an exceptional pair.
\end{lemma}

Combining it with Proposition~\ref{proposition:vb} we obtain the following

\begin{corollary}
\label{corollary:e0}
Let $X$ be a prime Fano threefold of genus~$12$ over a field~$\kk$ of characteristic~$0$.
There is a $2$-torsion class $\beta_E \in \Br(\kk)$ and a $\beta_E$-twisted rank-$2$ vector bundle~$E_0$ on~$X$ such that~$(E_0)_\bkk \cong E$,
where $E$ is the vector bundle from Lemma~\textup{\ref{lemma:vb-e}}.
\end{corollary}

Let $X$ be a prime Fano threefold of genus~$12$ over a field~$\kk$ of characteristic~$0$.
By Theorem~\ref{theorem-mukai} there is an embedding~$X_\bkk \hookrightarrow \Gr(3,V)$,
where $V$ is a vector space of dimension~7,
whose image is the zero locus of a section 
of the vector bundle $\wedge^2\cU^\vee \oplus \wedge^2\cU^\vee \oplus \wedge^2\cU^\vee$,
where $\cU$ denotes
the restriction to~$X_\bkk$ of the tautological rank-3 bundle on~$\Gr(3,V)$.
By~\cite[Theorem~2]{K97-V22} (see also~\cite{K96-V22}) 
there is a semiorthogonal decomposition
\begin{equation}
\label{eq:sod-x22}
\Db(X_\bkk) = \langle \cO, \cU^\vee, E, \wedge^2\cU^\vee \rangle,
\end{equation}
i.e., $(\cO, \cU^\vee, E, \wedge^2\cU^\vee)$
is a full exceptional collection in the bounded derived category of coherent sheaves~on~$X_\bkk$.
We use it to prove the following

\begin{proposition}
\label{proposition:x22-cu}
Let $X$ be a prime Fano threefold of genus~$12$ over a field $\kk$ of characteristic~$0$.
Then there exists an embedding $X \hookrightarrow \Gr(3,V)$ defined over~$\kk$ whose image is the zero locus of a section 
of the vector bundle $\wedge^2\cU^\vee \oplus \wedge^2\cU^\vee \oplus \wedge^2\cU^\vee$.
\end{proposition}

\begin{proof}
First, let us show that the exceptional collection~\eqref{eq:sod-x22} on~$X_\bkk$ is invariant under the action of the Galois group~$\Gal$.
Indeed, $\cO$ is evidently $\Gal$-invariant, and $E$ is $\Gal$-invariant by Lemma~\ref{lemma:vb-e}.
Therefore, for any $g \in \Gal$ the bundle $g^*\cU^\vee$ is exceptional and belongs to the subcategory
\begin{equation*}
\cA_{X_\bkk} := {}^\perp\cO \cap E^\perp \subset \Db(X_\bkk).
\end{equation*}
As it was observed in~\cite[Theorem~4.1]{K09-Fano} this category is equivalent to the derived category of representations
of the quiver with two vertices and three arrows.
Any exceptional object in this category can be obtained by mutations of the standard exceptional pair 
(formed by the two projective modules), which under the equivalence with~$\cA_{X_\bkk}$ correspond to the bundles~$V/\cU$ and~$\cU^\vee$.
Therefore, the ranks of exceptional objects in~$\cA_{X_\bkk}$ are given (up to sign) by the recursion
\begin{equation*}
r_{n+2} = 3r_{n+1} - r_n,
\qquad 
r_1 = 4,\quad r_2 = 3,\quad n \in \ZZ.
\end{equation*}
It is easy to see that $r_n > 3$ for $n \ne 2$, hence $\cU^\vee$ is the only exceptional vector bundle of rank~3 in~$\cA_{X_\bkk}$, 
and therefore 
\begin{equation*}
g^*\cU^\vee \cong \cU^\vee
\end{equation*}
for any $g \in \Gal$.

Now we apply Proposition~\ref{proposition:vb} to the rank-3 bundle~$\cU^\vee$.
We deduce that there is a Brauer class $\beta \in \Br(\kk)$ and a $\beta$-twisted vector bundle $\cE_0$ defined over~$\kk$ 
such that $\cU^\vee \cong (\cE_0)_\bkk$.
Since we have~\mbox{$\chi(\cU^\vee) = \dim H^0(X_\bkk,\cU^\vee) = 7$} it follows from~\eqref{eq:beta-chi} that $7\beta = 0$.
On the other hand, the line bundle~$\det(\cU^\vee) \cong \omega_{X_\bkk}^{-1}$ is defined over~$\kk$, hence~$3\beta = 0$ by~\eqref{eq:beta-rank}.
Therefore $\beta = 0$, hence $\cE_0$ is untwisted, i.e., $\cU$ is defined over~$\kk$.

Since $\cU_\bkk$ induces an embedding $X_\bkk \hookrightarrow \Gr_\bkk(3,7)$ whose image is the zero locus of a section 
of the vector bundle $\wedge^2\cU_\bkk^\vee \oplus \wedge^2\cU_\bkk^\vee \oplus \wedge^2\cU_\bkk^\vee$,
it follows that $\cU$ induces an embedding $X \hookrightarrow \Gr(3,7)$, and the image has the analogous description.
\end{proof}

The next results are not needed below, but we provide them for completeness.

\begin{corollary}
\label{corollary:x22-equivariant}
The embedding $X \hookrightarrow \Gr(3,V)$ constructed in Proposition~\xref{proposition:x22-cu}
is equivariant with respect to the action of~$\Aut(X)$.
\end{corollary}

\begin{proof}
The proof of Proposition~\ref{proposition:x22-cu} shows that the bundle~$\cU$ on~$X$ is invariant under the action of~$\Aut(X)$.
Therefore, it is equivariant for an appropriate central extension $G$ of~$\Aut(X)$
and the embedding~$X \hookrightarrow \Gr(3,V)$ is equivariant with respect to the action of~$G$.
But the action of the kernel of $G \to \Aut(X)$ on~$V$ is scalar, hence the action of $G$ on~$\Gr(3,V)$ factors through~$\Aut(X)$,
and hence the embedding $X \hookrightarrow \Gr(3,V)$ is equivariant with respect to~$\Aut(X)$.
\end{proof}

\begin{corollary}
\label{corollary:x22-db}
Let $X$ be a prime Fano threefold of genus~$12$ over a field $\kk$ of characteristic~$0$.
Then~$\Db(X) = \langle \Db(\kk), \Db(\kk), \Db(\kk,\beta_E), \Db(\kk) \rangle$, where $\beta_E \in \Br(\kk)$ is a $2$-torsion Brauer class.
\end{corollary}

\begin{proof}
By Proposition~\ref{proposition:x22-cu} the bundles $\cO$, $\cU^\vee$, and $\wedge^2\cU^\vee$ in~\eqref{eq:sod-x22} are defined over~$\kk$
and by Corollary~\ref{corollary:e0} the bundle~$E$ comes from a $\beta_E$-twisted vector bundle~$E_0$ defined over~$\kk$,
where $\beta_E$ is a $2$-torsion Brauer class.
The collection of vector bundles~$(\cO, \cU^\vee, E_0, \wedge^2\cU^\vee)$ gives the required semiorthogonal decomposition.
\end{proof}

\begin{remark}
One can construct an example of a smooth Fano threefold of genus~12 for which the Brauer class~$\beta_E$ is nontrivial.
\end{remark}

\subsection{Birational transformation for $\Gr(3,7)$}
\label{subsection:link-gr37-point}

Let $\bX = \Gr(3,V)$ be the Grassmannian of $3$-dimensional subspaces in a $7$-dimensional vector space~$V$.
Let~\mbox{$x_0 \in \bX$} be a point and let $U_0 \subset V$ be the corresponding $3$-dimensional subspace.
We denote by $\bH$ the Pl\"ucker class on~$\bX$ and by $\cU$ the tautological rank-3 bundle on~$\bX$.
Furthermore, we consider the $4$-dimensional vector space
\begin{equation*}
V^+ := V/U_0
\end{equation*}
and choose a splitting $V = U_0 \oplus V^+$.
Next, we consider the Grassmannian 
\begin{equation*}
\Gr(3,V^+) \cong \PP((V^+)^\vee) \cong \PP^3
\end{equation*}
and the tautological rank-3 bundle $\cU^+$ on it, and denote by $\bH^+$ its hyperplane class.
Note that
\begin{equation*}
V^+/\cU^+ \cong \cO(\bH^+).
\end{equation*}

We denote by $\Gr_{\Gr(3,V^+)}\bigl(3,U_0 \otimes \cO \oplus \cU^+\bigr)$ the relative Grassmannian of $3$-dimensional subspaces 
in the fibers of the rank-6 vector bundle $U_0 \otimes \cO \oplus \cU^+$ over $\Gr(3,V^+)$ 
and by~$\sigma_+$ its natural projection to~$\Gr(3,V^+)$.
Consider the morphism of vector bundles
\begin{equation}
\label{eq:x22-xi}
\xi \colon \cU \xhookrightarrow{\hspace{2em}} V \otimes \cO_\bX \xtwoheadrightarrow{\hspace{1.7em}} V^+ \otimes \cO_\bX,
\end{equation} 
defined as the composition of the tautological embedding with the natural projection.

\begin{proposition}
\label{proposition:g37-transformation}
There is a diagram of morphisms
\begin{equation*}
\xymatrix{
&
\Bl_\bZ(\bX) \ar[dl]_\sigma \ar@{=}[r] &
\Gr_{\Gr(3,V^+)}\bigl(3,U_0 \otimes \cO \oplus \cU^+\bigr) \ar[dr]^{\sigma_+}
\\
\bX &&& 
\Gr(3,V^+) \ar@{=}[r] & 
\PP^3,
}
\end{equation*}
where $\bZ = \fD_1(\xi) \subset \bX$ is the degeneracy locus of~$\xi$
and $\sigma$ is the blowup morphism.
\end{proposition}

\begin{proof}
Let $\cU' \subset U_0 \otimes \cO \oplus \cU^+$ denote the tautological rank-3 bundle on the relative Grassmannian~$\Gr_{\Gr(3,V^+)}\bigl(3,U_0 \otimes \cO \oplus \cU^+\bigr)$.
The natural embedding
\begin{equation*}
\cU' \xhookrightarrow{\hspace{2em}} U_0 \otimes \cO \oplus \cU^+ \xhookrightarrow{\hspace{2em}} (U_0 \otimes \cO) \oplus (V^+ \otimes \cO) = V \otimes \cO
\end{equation*}
induces a morphism 
\[
\sigma \colon \Gr_{\Gr(3,V^+)}\bigl(3,U_0 \otimes \cO \oplus \cU^+\bigr) \xlongrightarrow{\hspace{2em}} \Gr(3,V) = \bX
\]
such that $\cU'$ is the pullback of~$\cU$ (we will identify these bundles from now on).
We only need to show that this morphism is the blowup of the subscheme~$\bZ$.

For this note that the pair of morphisms $(\sigma,\sigma_+)$ gives an embedding
\begin{equation*}
\Gr_{\Gr(3,V^+)}\bigl(3,U_0 \otimes \cO \oplus \cU^+\bigr) \subset \Gr(3,V) \times \Gr(3,V^+)
\end{equation*}
and via this embedding $\Gr_{\Gr(3,V^+)}\bigl(3,U_0 \otimes \cO \oplus \cU^+\bigr)$ can be identified with the subscheme 
of pairs of $3$-dimensional subspaces $(U,U^+)$, where $U^+ \subset V^+$ and $U \subset U_0 \oplus U^+ \subset U_0 \oplus V^+ = V$. 
In other words, $\Gr_{\Gr(3,V^+)}\bigl(3,U_0 \otimes \cO \oplus \cU^+\bigr)$ is the zero locus for the section of the vector bundle 
\begin{equation*}
\cU^\vee \boxtimes \cO_{\Gr(3,V^+)}(\bH^+)
\end{equation*}
defined by~$\xi$.
By~\cite[Lemma~2.1]{K16} to deduce the claim we only need to check for each $k \ge 1$ 
that the $k$-th degeneracy locus~$\fD_k(\xi)$ of~$\xi$ has codimension at least~$k + 1$ (if non-empty).

Indeed, the first degeneracy locus~$\fD_1(\xi)$ parameterizes $3$-dimensional subspaces $U \subset V$ such that~\mbox{$\dim(U \cap U_0) \ge 1$}.
This locus is birational to a $\Gr(2,6)$-bundle over $\PP(U_0) = \PP^2$, hence has dimension~10 and codimension~$2$.
Similarly, the second degeneracy locus~$\fD_2(\xi)$ parameterizes subspaces $U \subset V$ such that $\dim(U \cap U_0) \ge 2$,
and is birational to a $\PP^4$-bundle over $\Gr(2,U_0) = \PP^2$, hence has dimension~6 and codimension~6.
Finally, the third degeneracy locus~$\fD_3(\xi)$ is the single point $x_0$, hence has codimension~12,
and for $k \ge 4$ we have $\fD_k(\xi) = \varnothing$.
\end{proof}

\begin{remark}
\label{remark:fiber-sigma}
The fiber of the morphism~$\sigma$ in Proposition~\xref{proposition:g37-transformation} over a point $[U] \in \Gr(3,7)$
is the projective space $\PP\left((V/(U + U_0))^\vee\right)$ of dimension~$\dim(U_0 \cap U)$.
\end{remark}

\subsection{The induced transformation of threefolds}
\label{subsection:v22-transformation}

Now let $X$ be a prime Fano threefold of genus~12 with a $\kk$-point~$x_0$.
By Proposition~\ref{proposition:x22-cu} the threefold $X$ is isomorphic to the zero locus of a global section 
of the bundle $(\wedge^2\cU^\vee)^{\oplus 3}$ on~$\bX$.
Let $B \subset \wedge^2V^\vee$ be the $3$-dimensional space of sections of~$\wedge^2\cU^\vee$ defining~$X$; then
\begin{equation}
\label{eq:x22-linear}
X = \bX \cap \PP\left(\Ker\bigl(\wedge^3V \xrightarrow{\hspace{1em}} B^\vee \otimes V\bigr)\right).
\end{equation}
Furthermore, the tangent space to~$X$ at $x_0$ is defined by the following exact sequence
\begin{equation*}
0 \xrightarrow{\hspace{1em}} T_{x_0}(X) \xrightarrow{\hspace{1em}} \Hom(U_0,V^+) \xrightarrow{\hspace{1em}} 
B^\vee \otimes \wedge^2U_0^\vee \xrightarrow{\hspace{1em}} 0,
\end{equation*}
where $\Hom(U_0,V^+) = T_{x_0}(\Gr(3,V))$ and the second arrow is the differential of the section defining~$X$.
In the next lemma we think of $T_{x_0}(X)$ as of a subspace in~$\Hom(U_0,V^+)$.

\begin{lemma}
\label{lemma:tangent-vectors}
The subset of $\PP(T_{x_0}(X))$ parameterizing vectors~$\tau$ of rank~$1$ is finite,
and the subset parameterizing vectors~$\tau$ of rank~$2$ is at most $1$-dimensional.
\end{lemma}

\begin{proof}
First, assume $\rank(\tau) = 1$.
If $V_2 := \Ker(\tau) \subset U_0$ and $V_4 \subset V$ corresponds to~\mbox{$\Ima(\tau) \subset V^+$}
then~$\tau$ is tangent to the line $\ell(V_2,V_4) \subset \Gr(3,V)$.
It follows from~\eqref{eq:x22-linear} that \mbox{$\ell(V_2,V_4) \subset X$}.
By~\cite[Lemma~2.1.8(ii)]{KPS} the scheme $\rF_1(X,x_0)$ is finite, 
hence the subset of~$\PP(T_{x_0}(X))$ parameterizing such~$\tau$ is finite.

Now assume $\rank(\tau) = 2$.
Set $V_1 := \Ker(\tau) \subset U_0$.
Then $\tau$ is tangent to the intersection of~$X$ with the sub-Grassmannian
\begin{equation*}
\Gr(2,V/V_1) \subset \Gr(3,V).
\end{equation*}
Let $W \subset V$ be the orthogonal of $V_1$ with respect to all $2$-forms in~$B$; then
\begin{equation*}
X \cap \Gr(2,V/V_1) \subset \Gr(2,W/V_1),
\end{equation*}
and, moreover, $X \cap \Gr(2,V/V_1)$ is a linear section of~$\Gr(2,W/V_1)$ of codimension at most~3.

Clearly, $\dim(W) \ge 4$.
Moreover, if $\dim(W) = 4$ then $\Ima(\tau) \subset W/U_0$, hence $\rank(\tau) = 1$, so we are in the previous case.

If $\dim(W) = 5$ then $\Gr(2,W/V_1) = \Gr(2,4)$, hence its linear section of codimension at most~3 is 
either a conic, or a plane, or a quadric of dimension~$\ge 2$.
But $X$ contains neither planes nor quadric surfaces, hence $X \cap \Gr(2,W/V_1)$ is a conic, and $\tau$ is tangent to it.
By~\cite[Lemma~4.2.6]{IP} the number of conics on~$X$ passing through~$x_0$ is finite;
each conic smooth at~$x_0$ has a unique tangent vector, 
while conics singular at~$x_0$ have a 1-dimensional set of tangent vectors in~$\PP(T_{x_0}(X))$.

Finally, if $\dim(W) = 6$ then $\dim(\Gr(2,W/V_1)) = 6$ and $\deg(\Gr(2,W/V_1)) = 5$.
Clearly, its linear section of codimension at most~3 cannot be contained in~$X$.
Similarly, $\dim(W)$ cannot be equal to~7.
This contradiction completes the proof.
\end{proof}

Consider the preimage of $X$ under the morphism~$\sigma$ from Proposition~\ref{proposition:g37-transformation}.

\begin{lemma}
\label{lemma:sigma-inverse-x}
The preimage $\sigma^{-1}(X)$ is a Cohen--Macaulay threefold with irreducible components
\begin{equation*}
\sigma^{-1}(X) = \tX_0 \cup \left( \bigcup_{[L] \in \rF_1(X,x_0)} \tX_L \right) \cup \Bl_Z(X) \subset X \times \PP^3,
\end{equation*}
where $Z = \bZ \cap X$ and
\begin{itemize}
\item 
the component $\tX_0$ is generically reduced and $(\tX_0)_{\mathrm{red}} = \{x_0\} \times \PP^3$,
\item 
$(\tX_L)_{\mathrm{red}} = L \times \PP^2 \subset X \times \PP^3$.
\end{itemize}
In particular, the component~$\tX_0$ has degree~$1$ over the second factor of~$X \times \PP^3$ 
and the components~$\tX_L$ do not dominate the second factor of~$X \times \PP^3$.
\end{lemma}
\begin{proof}
The proof of Proposition~\ref{proposition:g37-transformation} shows that
\begin{equation*}
\sigma^{-1}(X) \subset X \times \PP((V^+)^\vee) = X \times \PP^3
\end{equation*}
is the zero locus of a section of the rank-3 vector bundle $\cU^\vee \boxtimes \cO(\bH^+)$.
In particular, every its component has dimension at least~3.
To classify the components we consider the stratification of~$X$ by its intersections 
with the degeneracy loci $\fD_k(\xi) \subset \bX$ of the morphism~\eqref{eq:x22-xi}.

First, $X \cap \fD_3(\xi) = X \cap \{x_0\} = \{x_0\}$, and $\sigma^{-1}(x_0) \cong \PP^3$.
This provides the component~$\tX_0$; its scheme structure may be complicated, 
but we will show later that it is generically reduced.

Next, assume $[U] \in X \cap (\fD_2(\xi) \setminus \fD_3(\xi))$, i.e., $\dim(U \cap U_0) = 2$.
Set $V_2 := U \cap U_0$.
Note that the locus of all $U \subset V$ containing $V_2$ is the projective space $\PP(V/V_2) \cong \PP^4 \subset \bX$.
Since~$X$ is a linear section of $\bX$ by~\eqref{eq:x22-linear}, 
the intersection $\PP(V/V_2) \cap X$ is a linear subspace in $X$ 
containing the points~$x_0$ and $[U]$, hence containing the line joining them.
Since $X$ contains no planes, it follows that~$\PP(V/V_2) \cap X = L$, where $[L] \in \rF_1(X,x_0)$.
By Remark~\ref{remark:fiber-sigma} the map~\mbox{$\sigma^{-1}(L \setminus \{x_0\}) \to L \setminus \{x_0\}$} is a $\PP^2$-fibration, 
hence $\dim(\sigma^{-1}(L \setminus \{x_0\})) = 3$.
This provides the component~$\tX_L$.
More precisely, if~$L = \ell(V_2,V_4)$ for some~\mbox{$V_2 \subset V_4 \subset V$, $\dim V_i = i$, 
then \mbox{$(\tX_L)_{\mathrm{red}} = \PP(V_4/V_2) \times \PP(V_4^\perp) \cong L \times \PP^2$}.}
In particular, the image of~$\tX_L$ in $\PP^3$ is a plane.

Next, assume $[U] \in X \cap (\fD_1(\xi) \setminus \fD_2(\xi))$, i.e., $\dim(U \cap U_0) = 1$.
Set $V_1 := U \cap U_0$ and let~$W \subset V$ be the orthogonal of $V_1$ with respect to all $2$-forms in~$B$.
Then $[U] \in X \cap \Gr(2,W/V_1)$ and as in the proof of Lemma~\ref{lemma:tangent-vectors} 
we conclude that~$X \cap \Gr(2,W/V_1)$ is a conic or a line containing~$x_0$.
The case of lines was discussed above, so it remains to recall from~\cite[Lemma~4.2.6]{IP}
that the number of conics on~$X$ through~$x_0$ is finite, 
and for a conic $C \subset X$ the general fiber of~$\sigma$ over~$C$ is 1-dimensional by Remark~\ref{remark:fiber-sigma}.
Thus,
\begin{equation*}
\dim(\sigma^{-1}(C) \setminus (\tX_0 \cup (\cup \tX_L))) \le 2
\end{equation*}
hence conics do not contribute to irreducible components of~$\sigma^{-1}(X)$.

Finally, over $X \setminus \fD_1(\xi)$ 
the morphism $\sigma$ is an isomorphism, hence the strict transform of $X$, i.e., the blowup of $X$ along $Z = \bZ \cap X$,
is the last $3$-dimensional component of $\sigma^{-1}(X)$.

As we just showed, $\sigma^{-1}(X)$ has no components of dimension greater than~3, hence it is a Cohen--Macaulay threefold.

It remains to check that $\tX_0$ is generically reduced.
For this we consider the intersection of~$\sigma^{-1}(X)$ with the product of the infinitesimal neighborhood of~$x_0$ in~$X$ and~$\PP^3$.
This scheme can be written as the infinitesimal neighborhood of
\begin{equation*}
\{0\} \times \Gr(3,V^+) \subset
\{ (\tau,U^+) \in T_{x_0}(X) \times \Gr(3,V^+) \mid \Ima(\tau) \subset U^+ \}.
\end{equation*} 
We need to check that its projection to $\Gr(3,V^+)$ is an isomorphism over an open subset.
So, it is enough to check that the scheme 
\begin{equation*}
\label{eq:tau-up}
\{ (\tau,U^+) \in \PP(T_{x_0}(X)) \times \Gr(3,V^+) \mid \Ima(\tau) \subset U^+ \}
\end{equation*} 
is at most $2$-dimensional.
For this we consider the projection to~$\PP(T_{x_0}(X))$ and describe its fibers.

By Lemma~\ref{lemma:tangent-vectors} the locus of vectors $\tau \in \PP(T_{x_0}(X))$ of rank~1 is finite,
and the fibers over such vectors are $2$-dimensional.
Similarly, the locus of vectors $\tau \in \PP(T_{x_0}(X))$ of rank~$2$ is at most 1-dimensional,
and the fibers over such vectors are 1-dimensional.
Finally, the locus of vectors~\mbox{$\tau \in \PP(T_{x_0}(X))$} of rank~3 is $2$-dimensional
and the fibers over such vectors are 0-dimensional.
This finishes the proof.
\end{proof}

Now we can finally prove the theorem.

\begin{proof}[Proof of Theorem~\xref{theorem:v22}]
The component $\Bl_Z(X) \subset \sigma^{-1}(X)$ is birational to~$X$ by definition,
so it is enough to check that the morphism $\sigma_+ \colon \Bl_Z(X) \to \PP^3$ is birational.

First, note that every fiber of~$\sigma_+ \colon \sigma^{-1}(X) \to \PP^3$ is the zero locus of a section of~$\cU^\vee$ on~$X$.
Therefore, the general fiber is finite, since otherwise $\dim(\sigma^{-1}(X)) > 3$
which contradicts Lemma~\ref{lemma:sigma-inverse-x}.
On the other hand, if the zero locus of a section of $\cU^\vee$ is finite, its length is equal to~\mbox{$\rc_3(\cU^\vee) = 2$},
hence the general fiber of~$\sigma^{-1}(X)$ is a scheme of length~$2$.
Therefore, the morphism 
\begin{equation*}
\sigma_+ \colon \sigma^{-1}(X) \to \PP^3
\end{equation*}
is generically finite of degree~$2$.
Furthermore, by Lemma~\ref{lemma:sigma-inverse-x} the component~$\tX_0$ has degree~1 and the components~$\tX_L$ have degree~0 over~$\PP^3$.
This means that the remaining component~$\Bl_Z(X)$ has degree~1, hence the map $\sigma_+ \colon \Bl_Z(X) \to \PP^3$ is birational.
\end{proof}

\appendix

\section{Application to cylinders}
\label{subsection:cylinders}

Recall that a variety $X$ is {\sf cylindrical} if there is an open subset in~$X$ isomorphic to $U \times \mathbb{A}^1$.
Similarly, for any $r \ge 1$ we say that $X$ is {\sf $r$-cylindrical}
if there is an open subset in~$X$ isomorphic to $U \times \mathbb{A}^r$.
The existence of cylinders on projective varieties is an interesting question related
to the study of automorphism groups of affine cones~\cite{Kishimoto-Prokhorov-Zaidenberg-criterion}.

By~\cite{Prokhorov-Zaidenberg-Fano-Mukai} every Mukai fourfold of genus $g=10$ is $4$-cylindrical.
Moreover, there are families of cylindrical Mukai fourfolds of genus~$g \in \{7,\, 8,\, 9\}$,
see~\cite{Prokhorov-Zaidenberg-4-Fano,Prokhorov-Zaidenberg-Fano-4-new}.
Here we prove the following result.

\begin{proposition}
\label{proposition:cylinders}
Let~$X$ be a Mukai variety of genus~$g \in \{7,\, 8,\, 9,\, 10\}$ and dimension~$n \ge 5$.
If~$X(\kk) \ne \varnothing$ then~$X$ is $(n-4)$-cylindrical.
\end{proposition}

\begin{proof}
First, let $g \in \{7,\, 8,\, 10\}$.
It follows from the proof of Theorem~\ref{theorem:sections} that 
\begin{equation*}
X \setminus p_\cL(\cL(X,x_0)) = \tX^+ \setminus \PP_{\bX^+}(\cE_0).
\end{equation*}
Let $U := \bX^+ \setminus \fD(\xi) \subset \bX^+$ be the open subset where the morphism~$\xi_0$ (see~\eqref{eq:xi0}) is surjective,
so that~$\Ker(\xi\vert_U)$ and $\Ker(\xi_0\vert_U)$ are vector bundles of ranks~$n - 3$ and~$n - 4$, respectively. Then
\begin{equation*}
\sigma_+^{-1}(U) \setminus \PP_{\bX^+}(\cE_0)
\end{equation*}
is isomorphic to the total space of~$\Ker(\xi\vert_U)$.
Restricting to a smaller open subset $U' \subset U$ over which this bundle is trivial, 
we obtain an $(n-4)$-cylinder in~$\tX^+ \setminus \PP_{\bX^+}(\cE_0)$, hence in~$X$.

Now let $g=9$. 
If $n=6$, the proof of Theorem~\ref{theorem:lg36} (see Lemma~\ref{lemma:tilde-alpha})
shows that~\mbox{$X = \LGr(3,6)$} contains $\mathbb A^6$ as a Zariski open subset, hence $X$ is $6$-cylindrical in this case. 

Let $n=5$, so that $X$ is a hyperplane section of~$\LGr(3,6)$.
Consider the birational transformation of Theorem~\ref{theorem:mukai-9}.
Let $X^+ \subset \PP^6$ be the corresponding quadric and let $E^+ \subset X^+$ be the strict transform of the exceptional divisor~$E$.
Then $X^+ \setminus E^+$ is isomorphic to an open subset of~$X$, so it is enough to find a cylinder in~$X^+ \setminus E^+$.
But~$E^+$ is a hyperplane section of the quadric~$X^+$ containing a smooth $\kk$-rational point;
projecting from such a point one can easily find the required cylinder.
\end{proof}

\bibliography{rat}
\bibliographystyle{alpha}

\end{document}